\documentclass{amsart}

\usepackage{amsmath}
\usepackage{amsthm}
\usepackage{hyperref}
\usepackage{amsfonts,graphics,amsthm,amsfonts,amscd,latexsym}
\usepackage{epsfig}
\usepackage{flafter}
\usepackage{mathtools}
\usepackage{comment}
\usepackage{stmaryrd}

\usepackage{mathabx,epsfig}

\hypersetup{
    colorlinks=true,    
    linkcolor=blue,          
    citecolor=blue,      
    filecolor=blue,      
    urlcolor=blue           
}
\usepackage{tikz}
\usetikzlibrary{graphs,positioning,arrows,shapes.misc,decorations.pathmorphing}

\tikzset{
    >=stealth,
    every picture/.style={thick},
    graphs/every graph/.style={empty nodes},
}

\tikzstyle{vertex}=[
    draw,
    circle,
    fill=black,
    inner sep=1pt,
    minimum width=5pt,
]
\usepackage[position=top]{subfig}
\usepackage{amssymb}
\usepackage{color}

\calclayout

\usetikzlibrary{decorations.pathmorphing}
\tikzstyle{printersafe}=[decoration={snake,amplitude=0pt}]

\newcommand{\supp}{\operatorname{supp}}

\newcommand{\vol}{\operatorname{vol}}

\newcommand{\Aut}{\operatorname{Aut}}

\newcommand{\pp}{\mathbb{P}}

\newcommand{\qq}{\mathbb{Q}}
\newcommand{\zz}{\mathbb{Z}}
\newcommand{\nn}{\mathbb{N}}
\newcommand{\rr}{\mathbb{R}}
\newcommand{\cc}{\mathbb{C}}
\newcommand{\kk}{\mathbb{K}}

\def\O#1.{\mathcal {O}_{#1}}			
\def\pr #1.{\mathbb P^{#1}}				
\def\af #1.{\mathbb A^{#1}}			
\def\ses#1.#2.#3.{0\to #1\to #2\to #3 \to 0}	
\def\xrar#1.{\xrightarrow{#1}}			
\def\K#1.{K_{#1}}						
\def\bA#1.{\mathbf{A}_{#1}}			
\def\bM#1.{\mathbf{M}_{#1}}				
\def\bL#1.{\mathbf{L}_{#1}}				
\def\bB#1.{\mathbf{B}_{#1}}				
\def\bK#1.{\mathbf{K}_{#1}}			
\def\subs#1.{_{#1}}					
\def\sups#1.{^{#1}}						

\DeclareMathOperator{\coeff}{coeff}

\usepackage{tikz}
\usetikzlibrary{matrix,arrows,decorations.pathmorphing}

  \newtheorem{theorem}{Theorem}[section]
  \newtheorem{lemma}[theorem]{Lemma}
  \newtheorem{proposition}[theorem]{Proposition}

  \newtheorem{notation}[theorem]{Notation}
  \newtheorem{definition}[theorem]{Definition}
  \newtheorem{example}[theorem]{Example}

  \newtheorem{question}[theorem]{Question}

\newtheorem{remark}[theorem]{Remark}

\theoremstyle{remark}

\numberwithin{equation}{section}

\usepackage[all]{xy}

\begin{document}

\title[Polarized endomorphisms of log Calabi--Yau pairs]{Polarized endomorphisms of log Calabi--Yau pairs}

\author[J.~Moraga]{Joaqu\'in Moraga}
\address{UCLA Mathematics Department, Box 951555, Los Angeles, CA 90095-1555, USA
}
\email{jmoraga@math.ucla.edu}

\author[J.I.~Y\'a\~nez]{Jos\'e Ignacio Y\'a\~nez}
\address{UCLA Mathematics Department, Box 951555, Los Angeles, CA 90095-1555, USA
}
\email{yanez@math.ucla.edu}

\author[W. Yeong]{Wern Yeong}
\address{UCLA Mathematics Department, Box 951555, Los Angeles, CA 90095-1555, USA
}
\email{wyyeong@math.ucla.edu}

\thanks{The first author was partially supported by NSF research grant DMS-2443425.}


\begin{abstract}
Let $(X,\Delta)$ be a dlt log Calabi--Yau pair admitting a polarized endomorphism.
We show that $(X,\Delta)$ is a finite quotient of a toric log Calabi--Yau fibration over an abelian variety.
We provide an example
which shows that the previous statement does not hold if we drop the dlt condition of $(X,\Delta)$ even if $X$ is a smooth variety.
Given a klt type variety $X$
and a log Calabi--Yau pair $(X,\Delta)$
admitting a polarized endomorphism,
we show that a suitable birational modification of $(X,\Delta)$ is a finite quotient of a toric log Calabi--Yau fibration over an abelian variety.
\end{abstract}

\maketitle

\setcounter{tocdepth}{1} 
\tableofcontents

\section{Introduction}

A {\em polarized endomorphism} is an endomorphism $f\colon X\rightarrow X$ for which $f^*A\sim mA$ for some ample divisor $A$ on $X$ and $m\geq 2$.
In other words, in a suitable projective embedding of $X$ the endomorphism has degree at least $2$.
It is expected that very few smooth varieties carry polarized endomorphisms.
Two natural examples of smooth projective varieties that admit
polarized endomorphisms
are toric varieties and abelian varieties.
Some finite quotients of the aforementioned
examples admit polarized endomorphisms as well.
It is a folklore conjecture that klt type varieties admitting polarized endomorphisms 
are finite quotients
of toric fibrations
over abelian varieties.
We emphasize that the condition on the singularities is essential (see~\cite[Example 6.7]{MYY24}).
At the same time, it is expected that every variety $X$ that admits a polarized endomorphism is of log Calabi--Yau type, i.e., $X$ admits a boundary divisor $\Delta$ for which $(X,\Delta)$ is log canonical and $K_X+\Delta\equiv 0$ (see~\cite[Conjecture 1.2]{BG17}).
In this article, we study the folklore conjecture for log Calabi--Yau pairs $(X,\Delta)$.
More precisely, we say that $(X,\Delta)$ admits a polarized endomorphism if
$f\colon X\rightarrow X$ is a polarized endomorphism with $f^*(K_X+\Delta)=K_X+\Delta$.
The previous being an equality of divisors, not just of classes of divisors.
The first theorem of this article is the following.

\begin{theorem}\label{introthm:pol-endo-log-CY}
Let $(X,\Delta)$ be a dlt log Calabi--Yau pair.
Assume that $(X,\Delta)$ admits a polarized endomorphism.
Then $(X,\Delta)$ is a finite quotient of a toric log Calabi--Yau fibration over an abelian variety.
\end{theorem} 

A toric log Calabi--Yau fibration is a fibration from a log pair for which all fibers are toric varieties with their torus invariant divisor (see Definition~\ref{def:toric-log-CY-fib}).
Before discussing the main tools involved in the proof of Theorem~\ref{introthm:pol-endo-log-CY}, we recall some results in the literature about varieties admitting polarized endomorphisms.
The folklore conjecture is known in several cases: for smooth surfaces and almost all singular surfaces~\cite{Nak02, Nak21}, for smooth Fano $3$-folds~\cite{MZZ22}, for homogeneous varieties~\cite{PS89}, and for klt Calabi--Yau varieties~\cite{Men17}. 
The statement of Theorem~\ref{introthm:pol-endo-log-CY} is known for klt log Calabi--Yau pairs due to the work of Yoshikawa (see~\cite{Yos21}). In~\cite{KT23}, Kawakami and Totaro proved that a variety that admits a polarized endomorphism satisfies Bott vanishing.
In particular, a smooth Fano variety that
admits a polarized endomorphism is rigid (see~\cite[Remark 2.3]{KT23}).
A lot of work related to the folklore conjecture has been done under the hypothesis of the existence of a {\em completely invariant divisor}.
A completely invariant divisor for an endomorphism $f\colon X\rightarrow X$
is a divisor $\Delta$ on $X$ for which
$f^{-1}(\Delta)=\Delta$.
In~\cite{HN11}, Hwang and Nakayama proved that a Fano manifold $X$ of Picard rank one admitting a polarized endomorphism $f$ that is \'etale outside a completely invariant divisor
must be a projective space.
In~\cite{MZ20}, Meng and Zhang proved that if $X$
is a smooth rationally connected variety
admitting a polarized endomorphism that
is \'etale outside a completely invariant divisor $\Delta$, then $(X,\Delta)$ is toric.
In~\cite{MZ23}, Meng and Zhong proved that a pair $(X,\Delta)$, consisting of a smooth rationally connected variety and a reduced divisor, is a toric pair if and only if $X$ admits a 
polarized endomorphism that is \'etale outside $\Delta$.
In the three previous results~\cite{HN11,MZ20,MZ23},
the arguments used rely on the smoothness condition.
Note that in Theorem~\ref{introthm:pol-endo-log-CY}, we have an invariant divisor $\Delta$, but $f$ may not be \'etale on the complement of $\Delta$ (see~\cite[Example 6.2]{MYY24}).
In~\cite[Theorem 1]{MYY24}, the authors proved a version of Theorem~\ref{introthm:pol-endo-log-CY} for $X$ a Fano type variety and $\Delta$ a reduced divisor.
Both conditions are necessary for the proof provided in~\cite{MYY24}. The reducedness of $\Delta$ is used to approximate it using divisors with standard coefficients, while
the Fano type condition is used via the Jordan property for fundamental groups (see, e.g.,~\cite{BFMS20,Mor21}). 

To prove Theorem \ref{introthm:pol-endo-log-CY} we use that one can lift a surjective endomorphism to a small $\qq$-factorial modification.

\begin{theorem}[cf. Theorem \ref{thm:lift-to-Q-fact}]
Let $X$ be a projective klt type variety and $f\colon X \to X$ a surjective endomorphism.
Then, after possibly replacing $f$ with a higher iteration, there exists an $f$-equivariant $\qq$-factorial klt model.
\label{introthm:lift-to-Q-fact}
\end{theorem}

Such model depends on $f$. In particular, one can drop the $\qq$-factorial condition from \cite[Theorem 1.7]{MZ22} (see also \cite{MZ24}).

In Example~\ref{ex:not-pdlt}, we show that Theorem~\ref{introthm:pol-endo-log-CY} is not valid
if we drop the dlt condition 
for the log Calabi--Yau pair $(X,\Delta)$ even if $X$ is a smooth variety.
In Definition~\ref{def:pdlt}, we introduce the concept of {\em partially dlt} singularities.
We abbreviate this concept as pdlt.
Roughly speaking, a pair $(X,\Delta)$ is pdlt if
every log canonical center of $(X,\Delta)$ is contained in the support of $\lfloor \Delta\rfloor$.
Theorem~\ref{thm:main-pdlt} shows that Theorem~\ref{introthm:pol-endo-log-CY} is still valid when $(X,\Delta)$ is a pdlt 
log Calabi--Yau pair admitting an int-amplified endomorphism.
Let us point out that the log Calabi--Yau pairs considered in~\cite[Theorem 1]{MYY24} are pdlt.
Thus, Theorem~\ref{thm:main-pdlt} generalizes the previous work by the authors
to the setting in which the ambient variety $X$ is
not necessarily a Fano variety.

We show that some iteration of a polarized endomorphism 
of a log Calabi--Yau pair lifts to a suitable pdlt modification, i.e., it lifts to a higher birational model that has pdlt singularities (see Definition~\ref{def:pdlt-mod}).

\begin{theorem}[cf. Theorem~\ref{thm:lifting-pdlt}]\label{introthm:lifting-pol}
Let $X$ be a klt type variety
and let $(X,\Delta)$ be a log Calabi--Yau pair admitting a polarized endomorphism $f$.
Then, some iteration of $f$ lifts to a pdlt modification of $(X,\Delta)$.
\end{theorem}

Putting Theorem~\ref{introthm:lifting-pol} and Theorem~\ref{thm:main-pdlt} together, we conclude the following.

\begin{theorem}\label{introthm:3}
Let $X$ be a klt type variety and let $(X,\Delta)$ be a log Calabi--Yau pair admitting a polarized endomorphism $f$.
Then, the pair $(X,\Delta)$ admits a pdlt modification that is a finite quotient
of a toric log Calabi--Yau fibration
over an abelian variety.
\end{theorem}

Example~\ref{ex:not-pdlt} shows
that in general both kind
of modifications:
the pdlt modification
and the finite cover
are needed in the statement 
of Theorem~\ref{introthm:3}.
We point out that in the statement of Theorem~\ref{introthm:3}, we may need to take a pdlt modification even 
if $X$ is a toric variety itself.
Indeed, the boundary divisor $\Delta$ may not be torus invariant, so the structure morphism of $X$ is not a toric log Calabi--Yau fibration.

\subsection{Sketch of the proofs}
In this subsection, we provide a sketch of the proof of the two main theorems.

First, we explain the ideas that lead to 
Theorem~\ref{introthm:pol-endo-log-CY}
and 
Theorem~\ref{thm:main-pdlt}.
Let $(X,\Delta)$ be a pdlt log Calabi--Yau pair that admits a polarized endomorphism.
In Lemma~\ref{lem:stand-coeff}, we prove that $\Delta$ must have standard coefficients.
Thus, to turn $\Delta$ into a reduced divisor, 
we take the index one cover of the torsion divisor $K_X+\Delta$.
The advantage of pdlt singularities
is that they are preserved under taking index one covers (see Lemma~\ref{lem:finite-cover-pdlt}).
In particular, the index one cover will have klt type singularities.
Using~\cite[Corollary 3.2]{MYY24}, we lift the polarized endomorphism to this index one cover.
Using Lemma~\ref{lem:passing-to-small-Q-fact} and Theorem~\ref{thm:lift-to-Q-fact}, we show that it suffices to prove the theorem for a small $\qq$-factorialization of the index one cover of $K_X+\Delta$.

From now on, we may assume that $X$ is $\qq$-factorial, $(X,\Delta)$ is pdlt, and $K_X+\Delta\sim 0$.
Using the work of Yoshikawa~\cite[Theorem 6.6]{Yos21}, 
we obtain a Fano type fibration $\phi\colon (X,\Delta) \rightarrow A$
to an abelian variety $A$
that commutes with the polarized endomorphism $f$.
The endomorphism $f$ may not induce a polarized endomorphism on general fibers. 
Thus, we cannot directly apply~\cite[Theorem 1]{MYY24} to conclude that a general fiber $(F,\Delta_F)$ is a toric log Calabi--Yau pair.
Instead, we use the Jordan property of
the fundamental group of $(F,\Delta_F)$ to conclude that every iteration of
$f$ induces Galois crepant finite quotients
between the very general fibers of $(X,\Delta)\rightarrow A$.
Thus, for a very general fiber $(F,\Delta_F)$, 
we have an infinite nested sequence $\{G_i\}_{i\in \zz_{\geq 0}}$ of finite subgroups such that $G_i\leqslant {\rm Aut}(F,\Delta)$
and the quotients $(F/G_i,\Delta/G_i)$ belong to a bounded family. Indeed, the quotients are fibers of $\phi$.
In Section~\ref{sec:bounded-finite-quot}, we study bounded finite quotients of log Calabi--Yau pairs, to show that $(F,\Delta_F)$ must be toric, after possibly taking a finite cover (see Theorem~\ref{thm:toric-bounded-quot}).
By Lemma~\ref{lem:toric-vg-g}, the general fiber of $\phi$ is also a toric log Calabi--Yau pair.

Thus, we have a log Calabi--Yau fibration
$\phi\colon (X,\Delta)\rightarrow A$
for which the general fiber $(F,\Delta_F)$ is a toric log Calabi--Yau pair.
Finally, we are left with the problem of showing that $\phi\colon (X,\Delta)\rightarrow A$ is formally toric over any closed point of the base.
Using the relative complexity (see~\cite[Theorem 1]{MS21}) it suffices to show that $\rho(X/A)=\rho(F)$.
In order to prove this, we need to argue that $X\rightarrow A$ has no degenerate divisors.
For this purpose, we show that there is an induced polarized endomorphism $\phi_A\colon A\rightarrow A$
and that every degenerate divisor of $\phi$ maps onto a fixed subvariety for $\phi_A$ (see Lemma~\ref{lem:degenerate-divisors-pullback}).
However, polarized endomorphisms on abelian varieties have no fixed subvarieties (see, e.g.,~\cite[Lemma 4.7]{MZ18}).
This finishes the sketch of the proof in the case of polarized endomorphisms.
The case of int-amplified endomorphisms brings some more technical difficulties.

Now, we turn to explain the ideas behind the proof of Theorem~\ref{introthm:lifting-pol}.
Let $f\colon (X,\Delta)\rightarrow (X,\Delta)$ be a polarized endomorphism of a log Calabi--Yau pair $(X,\Delta)$.
The finite map $f$ may not lift to a dlt modification of $(X,\Delta)$.
However, given a dlt modification $(Y,\Delta_Y)$ of $(X,\Delta)$, we may find a qdlt modification $(W,\Delta_W)$ of $(X,\Delta)$ such that $f$ lifts to
$g\colon (W,\Delta_W) \rightarrow (Y,\Delta_Y)$ (see Theorem~\ref{thm:def-D(f)}). 
This lifting allows us to construct a continuous map $\mathcal{D}(f) \colon \mathcal{D}(X,\Delta) \rightarrow \mathcal{D}(X,\Delta)$ between pseudomanifolds that describes how $f$ acts on the log canonical places of $(X,\Delta)$. Given a log canonical center $Z$ of $(X,\Delta)$, at which the pair is not pdlt, our aim is to extract a log canonical place $E$ with center $Z$. However, to lift the polarized endomorphism $f$ to a model extracting $E$, we need to show that $E$ is fixed by $f$, or equivalently, that the vertex $v_E \in \mathcal{D}(X,\Delta)$ is fixed by $\mathcal{D}(f)$.
In order to show this, 
we show that $\mathcal{D}(X,\Delta)$ is locally collapsible over $X$, i.e., 
the sub-complex of $\mathcal{D}(X,\Delta)$ consisting of log canonical places mapping to the same center $Z$ in $X$ is collapsible (see Lemma~\ref{lem:collapsible-local}).
In Theorem~\ref{thm:D(f)-bijection}, we show that $\mathcal{D}(f)$ is a bijection in our setting.
Thus, using the local collapsibility and the bijectivity, we show that $f$ has a fixed log canonical place $E$ over every log canonical center $Z$ of $(X,\Delta)$.
By inductively extracting log canonical places over non-pdlt centers of $(X,\Delta)$, we lift the polarized endomorphism to a pdlt modification.

Theorem~\ref{introthm:3} is a straightforward application of Theorem~\ref{introthm:lifting-pol} 
and Theorem~\ref{thm:main-pdlt}.

\subsection*{Acknowledgements}
The authors would like to thank Sheng Meng, Noboru Nakayama, Burt Totaro, and Guolei Zhong for many useful discussions and comments. 

\section{Preliminaries}

In this section, we prove some preliminary results about
toric fibrations, singularities, log Calabi--Yau pairs, and polarized endomorphisms.
We work over an uncountable algebraically closed field $\kk$ of characteristic zero.
The set of {\em standard coefficients} is the set $\{1-\frac{1}{n} \mid n \in \zz_{>0}\}\cup\{1\}$. 
We say that a pair $(X,\Delta)$ has {\em standard coefficients} if the coefficients
of $\Delta$ are in the set of standard coefficients.
A sequence of groups $\{G_i\}_{i\in \zz}$ is said to be {\em nested} if $G_i\leq G_j$ for $i\leq j$. The {\em multiplicity} of an algebraic variety $X$ at a closed point $x\in X$ is denoted by $\mu(X;x)$.

\subsection{Singularities}
In this subsection, we recall the concepts of singularities of the Minimal Model Program.
Further, we introduce the concept of partially dlt singularities.

\begin{definition}
{\em 
Let $(X,\Delta)$ be a log pair, i.e. $X$ is a normal variety and $\Delta$ is a $\qq$-divisor such that $K_X+\Delta$ is $\qq$-Cartier.
Let $\pi\colon Y\rightarrow X$ be a projective birational morphism from a normal variety $Y$,  and let $E\subset Y$ be a prime divisor.
The {\em log discrepancy}, $a_E(X,\Delta)$, of $(X,\Delta)$ at $E$ is 
\[
1-{\rm coeff}_E(\Delta_Y) 
\text{ where } 
K_Y+\Delta_Y = \pi^*(K_X+\Delta). 
\]
We say that $(X,\Delta)$ is {\em Kawamata log terminal} (klt for short) if all its log discrepancies are positive.
We say that $(X,\Delta)$ is {\em log canonical} (lc for short)
if all its log discrepancies are non-negative.
A log canonical pair $(X,\Delta)$ is called a {\em log Calabi--Yau pair} if $K_X+\Delta \sim_\qq 0.$
We say that $X$ is of {\em klt type} if there exists a boundary $\Delta$ on $X$ such that $(X,\Delta)$ is klt.
}
\end{definition}

\begin{definition}
{\em 
Let $(X,\Delta)$ be a log canonical pair.
A {\em log canonical place} of $(X,\Delta)$
is a divisor $E$ over $X$ for which 
$a_E(X,\Delta)=0$.
A {\em log canonical center} of $(X,\Delta)$
is the image on $X$ of a log canonical place. 
We often abbreviate log canonical center by lcc.
}
\end{definition} 

\begin{definition}\label{def:pdlt-sing}
{\em  
Let $(X,\Delta)$ be a log canonical pair.
We say that $(X,\Delta)$ is {\em divisorially log terminal} (or dlt) if there exists an open $U\subseteq X$ satisfying the following conditions:
\begin{itemize}
\item the pair $(X,\Delta)$ is simple normal crossing on $U$, and 
\item every log canonical center of $(X,\Delta)$
intersects $U$
and is a strata of $\lfloor \Delta \rfloor$.
\end{itemize}
}
\end{definition}

\begin{definition}\label{def:qdlt}
{\em 
Let $(X,\Delta)$ be a log canonical pair, and 
let $Z$ be a log canonical center of $(X,\Delta)$.
We say that $(X,\Delta)$ is {\em quotient dlt} (or qdlt for short) at $Z$ if the following condition holds:
There is a simple normal crossing pair
$(Y,B_1+\dots+B_k)$ and a finite abelian group $A$ acting on $Y$, preserving each of the divisors $B_i$'s, such that 
$(X,\Delta)\simeq (Y,B_1+\dots+B_k)/A$ holds on a neighborhood of the generic point of $Z$.
We say that a pair $(X,\Delta)$ is {\em qdlt} if it is qdlt at every log canonical center of $(X,\Delta)$.
}
\end{definition} 

\begin{definition}\label{def:pdlt}
{\em 
Let $(X,\Delta)$ be a log canonical pair.
We say that $(X,\Delta)$ is {\em partially divisorially log terminal} (or pdlt) if $(X,\Delta-\epsilon S)$ is klt for every $\epsilon$ small enough
and $S$ an effective divisor with $\supp S = \supp \lfloor \Delta \rfloor$.
A {\em non-pdlt center} is a log canonical center of $(X,\Delta)$ which is not contained in the support of $\lfloor \Delta \rfloor$.
}
\end{definition} 

It is clear from the definition that a $\qq$-factorial dlt pair $(X,\Delta)$ is pdlt 
by taking $S=\lfloor \Delta\rfloor$.
Furthermore, a pdlt pair is of klt type
by definition. 

\begin{definition}\label{def:pdlt-mod}
{\em 
Let $(X,\Delta)$ be a log canonical pair.
We say that a projective birational morphism $p\colon Y\rightarrow X$ is a {\em pdlt modification} of $(X,\Delta)$ if it only extracts log canonical places of $(X,\Delta)$ and the pair $(Y,\Delta_Y)$ is pdlt, where $K_Y+\Delta_Y=p^*(K_X+\Delta)$.
}
\end{definition}

As a dlt pair is pdlt, 
every log canonical pair admits a pdlt modification (see, e.g.,~\cite[Theorem 3.1]{KK10}).
However, there are pdlt modifications
that are not dlt.
The following lemma is straightforward from the definition.

\begin{lemma}\label{lem:pdlt-vs-klt-type}
Let $X$ be a $\qq$-factorial variety
and $(X,\Delta)$ be a lc pair. 
Then, $(X,\Delta)$ is pdlt if and only if
$(X,\Delta-\epsilon \lfloor \Delta\rfloor)$ is klt for every $\epsilon>0$ small enough.
\end{lemma}

Now, we turn to prove the following lemma regarding finite covers of pdlt pairs.

\begin{lemma}\label{lem:finite-cover-pdlt}
Let $(X,\Delta)$ be a pdlt pair.
Let $f\colon Y\rightarrow X$ be a finite morphism
such that $K_Y+\Delta_Y=f^*(K_X+\Delta)$. 
Assume that $\Delta_Y\geq 0$. 
Then, the pair $(Y,\Delta_Y)$ is pdlt.
\end{lemma} 

\begin{proof}
If $P$ is a prime component of $\lfloor \Delta \rfloor$, then 
every component $f^{-1}(P)$ appears  in $\lfloor \Delta_Y\rfloor$ by Riemann-Hurwitz.
Vice-versa, by Riemann-Hurwitz, if $Q$ is a prime component of $\lfloor \Delta_Y\rfloor$, then $f(Q)$ is a prime component of $\lfloor \Delta\rfloor$.
In conclusion, we have $f^{-1}\lfloor \Delta\rfloor=\lfloor \Delta_Y \rfloor$.
Let $S$ be an effective divisor
for which $\supp S=\supp \lfloor \Delta \rfloor$
and $(X,\Delta-\epsilon S)$ is klt for $\epsilon>0$ small enough.
As finite pull-backs of klt pairs are sub-klt, we have that the pair $(Y,\Delta_Y - \epsilon f^*S)$ is klt for $\epsilon$ small enough.
Since $\supp f^*S =\supp \lfloor \Delta_Y\rfloor$, we conclude that $(Y,\Delta_Y)$ is pdlt.
\end{proof} 

\begin{lemma}\label{lem:pdlt-descends-finite}
Let $(X,\Delta)$ be a log pair with standard coefficients. Write $\Delta=\sum_{P\subset X}\left(1-\frac{1}{n_P}\right)P$ where $n_P$ is a positive integer and the sum runs over all the prime divisors $P$ of $X$.
Let $f\colon Y\rightarrow X$ be a finite morphism 
such that for every prime divisor $Q\in Y$ the ramification index $r_Q$ of $f$ at $Q$ divides $n_{f(Q)}$.
Let $K_Y+\Delta_Y=f^*(K_X+\Delta)$ such that $(Y,\Delta_Y)$ is a log pair.
If $(Y,\Delta_Y)$ has pdlt singularities, then so does $(X,\Delta)$. 
\end{lemma}

\begin{proof}
Let $h\colon Z\rightarrow X$ be the Galois closure of $f$. The condition on the ramification indices of $f$ implies that $\Delta_Z\geq 0$ where
$K_Z+\Delta_Z=h^*(K_X+\Delta)$.
By Lemma~\ref{lem:finite-cover-pdlt}, we conclude that $(Z,\Delta_Z)$ has pdlt singularities.
Let $G$ be the Galois group associated to $h$.
Then, the boundary divisor $\Delta_Z$ is $G$-invariant
and so $\lfloor \Delta_Z \rfloor$ is $G$-invariant.
As $(Z,\Delta_Z)$ is pdlt, there exists an effective divisor $S_Z$ with $\supp S_Z=\supp \lfloor \Delta_Z\rfloor$ and 
$(Z,\Delta_Z-\epsilon S_Z)$ is klt for $\epsilon$ small enough.
Replacing $S_Z$ with $\sum_{g\in G} g^*S_Z/|G|$, we may assume that $S_Z$ is $G$-invariant.
In particular, we can find an effective divisor 
$S$ on $X$ such that $h^*S=S_Z$.
Note that $\supp h^*\lfloor \Delta \rfloor =\supp \lfloor \Delta_Z\rfloor$. 
Therefore, we conclude that $\supp S=\supp \Delta$.
Henceforth, the pair $(X,\Delta-\epsilon S)$ is klt for $\epsilon$ small enough.
\end{proof}

\begin{definition}
{\em 
Let $f\colon Y\rightarrow X$ be a finite morphism.
Let $(Y,\Delta_Y)$ and $(X,\Delta)$ be two pairs.
We say that $f$ is a {\em crepant finite morphism} for the pairs if 
$f^*(K_X+\Delta)=K_Y+\Delta_Y$.
In this setting, we may also say that
$f\colon (Y,\Delta_Y)\rightarrow (X,\Delta)$
is a {\em crepant finite morphism}.
}
\end{definition}

\subsection{Endomorphisms} 
We recall the definitions of polarized and int-amplified endomorphisms on pairs and a key result from \cite{MYY24} that allows us to lift some iteration of such an endomorphism to a finite cover.

\begin{definition}
{\em 
An endomorphism $f\colon X\rightarrow X$ is a \emph{polarized endomorphism} if $f^*A\sim mA$ for some $m\geq 2$ and ample divisor $A$.
An endomorphism $f\colon X\rightarrow X$ is said to be {\em int-amplified} if $f^*A-A$ is ample
for some ample Cartier divisor $A$ in $X$.
}
\end{definition}

\begin{definition}
{\em
We say that a pair $(X,\Delta)$ 
admits a \emph{polarized endomorphism} if there is a polarized endomorphism $f: X \rightarrow X$ such that $f^*(K_X+\Delta)=K_X+\Delta.$
We say that a pair $(X,\Delta)$ admits a {\em int-amplified endomorphism} if there is an int-amplified endomorphism $f\colon X\rightarrow X$ such that
$f^*(K_X+\Delta)=K_X+\Delta$.
}
\end{definition}

Throughout this paper we use ~\cite[Theorem 3.1]{MYY24} several times.
We include its statement for readers convenience.

\begin{theorem}
Let $(X,\Delta)$ be a log Calabi--Yau pair, with $K_X + \Delta \sim 0$. Let $U:= X^{\rm{reg}}\setminus \Delta$.
Suppose that $(X,\Delta)$ admits an int-amplified endomorphism $f \colon (X,\Delta) \to (X,\Delta)$. 
Let $g\colon (Y,\Delta_Y) \to (X,\Delta)$ be a finite cover such that $g^{-1}(U) \to U$ is \'etale.
Then there exists $m,n \gg 0$,
a finite cover $h\colon (Z,\Delta_Z) \to (Y,\Delta_Y)$, 
and an int-amplified endomorphism $f_Z:(Z,\Delta_Z)\to (Z,\Delta_Z)$
such that the following diagram commutes
\[
\xymatrix{
	(Z,\Delta_{Z}) \ar[r]^{f_{Z}} \ar[d]_{g'} & (Z,\Delta_Z) \ar[r]^{h} \ar[d]^{g'} & (Y,\Delta_Y) \ar[d]^g\\
	(X,\Delta) \ar[r]_{f^n} & (X,\Delta) \ar[r]_{f^m} & (X,\Delta).
}
\]
\end{theorem}

\subsection{Toric fibrations} In this subsection, we prove some results regarding toric log Calabi--Yau fibrations.

\begin{definition}
{\em 
A {\em contraction} is a morphism
$f\colon X\rightarrow Y$ for which
$f_*\mathcal{O}_X=\mathcal{O}_Y$.
A {\em fibration} is a contraction
with positive dimension general fiber.
}
\end{definition}

\begin{definition}{\em 
Let $\phi\colon X\rightarrow Z$ be a fibration and $(X,\Delta)$ be a pair.
We say that $(X,\Delta)$ is {\em log Calabi--Yau} over $Z$ if 
$(X,\Delta)$ is log canonical and $K_X+\Delta \sim_{Z,\qq} 0$.
}
\end{definition} 

\begin{definition}
{\em 
Let $(X,\Delta)$ be a log pair
and $X\rightarrow Z$ be a projective fibration.
We say that $(X,\Delta)\rightarrow Z$ is a {\em weak toric log Calabi--Yau fibration}
if the following conditions hold:
\begin{enumerate}
\item the pair $(X,\Delta)$ is log Calabi--Yau over $Z$,
\item we have $K_X+\Delta\sim_Z 0$, 
\item the general fiber $F$ of $X\rightarrow Z$ is a projective toric variety, and 
\item the prime components of $\Delta$ restrict to the prime torus invariant divisors of $F$.
\end{enumerate} 
}
\end{definition} 

\begin{definition}\label{def:toric-log-CY-fib}
{\em 
Let $(X,\Delta)$ be a log pair
and $X\rightarrow Z$ be a projective fibration.
We say that $\phi\colon (X,\Delta)\rightarrow Z$ is a {\em toric log Calabi--Yau fibration} if 
the following conditions hold:
\begin{enumerate}
\item $\phi$ is a weak toric log Calabi--Yau fibration, and 
\item $\phi$ is a formally toric morphism 
over any closed point $z\in Z$.
\end{enumerate}
}
\end{definition} 

We will need the following definition.

\begin{definition}
{\em 
Let $\phi\colon X\to Y$ be a surjective morphism of normal projective varieties, and let $D$ be an effective divisor on $X$. 
\begin{itemize}
\item We say that $D$ is {\em $\phi$-exceptional} if $\supp(\phi(D))$ has codimension at least two in $Y$.
\item We say that $D$ is {\em $\phi$-degenerate} if $\supp(\phi(D))$ has codimension one in $Y$, and there is a prime divisor $E\not\subset \supp(D)$ such that $\supp(\phi(E))= \supp(\phi(D))$.
\end{itemize}
In any of the previous two cases, we say that the divisor $D$ is {\em degenerate}.
}
\end{definition}

The following proposition 
allows us to determine whether
a weak toric log Calabi--Yau fibration
is indeed a toric log Calabi--Yau fibration. 

\begin{proposition}\label{prop:from-weak-to-toric}
Let $\phi\colon (X,\Delta)\rightarrow Z$ be a weak toric log Calabi--Yau fibration.
Assume the following conditions hold:
\begin{enumerate}
\item the variety $X$ is $\qq$-factorial, 
\item $\phi$ is a Fano type morphism,
\item $\phi$ has no degenerate divisors,
\item $Z$ is a smooth variety, and 
\item both the moduli divisor $\mathbf{M}_Z$ and the boundary divisor $B_Z$ induced on $Z$ by the canonical bundle formula applied to $(X,\Delta)$ are trivial.
\end{enumerate} 
Then, the morphism $\phi\colon (X,\Delta)\rightarrow Z$ is a toric log Calabi--Yau fibration.
\end{proposition} 

\begin{proof}
Since $\phi$ is a Fano type morphism, 
we know that ${\rm Cl}(X/Z)$ is finitely generated.
Hence, there is an open subset $U\subseteq Z$
such that a divisor $D$ on $X_U$ is trivial over $Z$ if and only if $D$ is trivial on a closed fiber of $X_U\rightarrow U$.
In particular, we have that $\rho(X_U/U)\leq \rho(F)$ where $F$ is a general fiber.
Furthermore, the variety $X$ has no degenerate divisors over $Z$.
Thus, we get $\rho(X/Z)\leq \rho(F)$.
Indeed, all the divisors on the complement of $X_U$ are trivial over $Z$.
Note that these divisors are $\qq$-Cartier
by the $\qq$-factoriality assumption.
By condition (4) in the definition of weak toric log Calabi--Yau fibrations, we know that the number of prime components of $\Delta_{\rm hor}$ is at least $\dim(F)+\rho(F)$.
Thus, we conclude that the number of prime components of $\Delta_{\rm hor}$ is at least $\dim(F)+\rho(X/Z)$.

Let $z\in Z$ be a closed point.
Let $d:=\dim Z$ and 
$H_1+\dots+H_d$ be a reduced simple normal crossing divisor through $z$.
Then, the pair $(Z,H_1+\dots+H_d;z)$ is log canonical.
Hence, the pair
$(X,\Delta+\phi^*H_1+\dots+\phi^*H_d)$ is log Calabi--Yau over $z\in Z$.
Indeed, by condition (5) in the statement of the proposition, we know that the log Calabi--Yau pair
$(X,\Delta+\phi^*H_1+\dots+\phi^*H_d)$
induces $(Z,H_1+\dots+H_d)$ by the canonical bundle formula.
Note that
\[
c_z(X/Z,\Delta+\phi^*H_1+\dots+\phi^*H_d)=
\dim X + \rho(X/Z) - |\Delta + \phi^*H_1+\dots+\phi^*H_d| \leq 
\]
\[
(\dim F +\rho(F) - |\Delta_{\rm hor}|) + d -
|\phi^*H_1+\dots+\phi^*H_d| \leq 0.
\]
By~\cite[Theorem 1]{MS21}, we conclude that $\phi\colon (X,\Delta)\rightarrow Z$ is formally toric over $z\in Z$. 
\end{proof}

\begin{lemma}\label{lem:toric-vg-g}
Let $\phi\colon (X,\Delta)\rightarrow Z$ 
be a log Calabi--Yau fibration.
If a very general fiber of $\phi$ is a toric log Calabi--Yau pair,
then the general fiber of $\phi$ is a toric
log Calabi--Yau pair.
\end{lemma} 

\begin{proof}
Let $f$ be the dimension of a general fiber.
Let $(X_z,\Delta_z)$ be a very general fiber
and $(X_{\bar{\eta}},\Delta_{\bar{\eta}})$ 
be the geometric generic fiber.
By~\cite[Lemma 2.1]{Via13}, there is an abstract isomorphism
$(X_z,\Delta_z)\simeq (X_{\bar{\eta}},\Delta_{\bar{\eta}})$
that depends on the choice of an isomorphism of fields
$\cc\rightarrow \bar{\eta}$.
In particular, we get 
$\mathbb{G}_{m,\bar{\eta}}^f \leqslant {\rm Aut}(X_{\bar{\eta}},\Delta_{\bar{\eta}})$.
We conclude that ${\rm Aut}(X_\eta,\Delta_\eta)$ contains a torus $\mathbb{T}_\eta$ of rank at least $f$.
Therefore, over an affine $U$ of the base $Z$, we
have that ${\rm Aut}(X_U,\Delta_U)$ contains a torus $\mathbb{T}_U$ of rank at least $f$.
Further, we may assume that every fiber over $U$ is a log Calabi--Yau pair.
Thus, for every $u\in U$, the log Calabi--Yau pair
$(X_u,\Delta_u)$ satisfies that $\mathbb{G}_m^f\leqslant{\rm Aut}(X_u,\Delta_u)$. 
Hence, the log Calabi--Yau pair $(X_u,\Delta_u)$ is toric for every $u\in U$.
\end{proof}

\begin{lemma}\label{lem:div-hor}
Let $\phi\colon (X,\Delta)\rightarrow Z$ be a log Calabi--Yau fibration with $K_X+\Delta\sim 0$.
There exists a commutative diagram
\[
\xymatrix{
(X,\Delta)\ar[d]_-{\phi} & (X',\Delta')\ar[l]_-{f}\ar[d]^-{\phi'} \\
Z & Z' \ar[l]^-{f_Z}
}
\]
where $f$ and $f_Z$ are crepant finite morphisms,
$\phi'$ is a log Calabi--Yau fibration
with $K_{X'}+\Delta' \sim 0$, 
and the restriction of every stratum of 
$\Delta'$ to a general fiber of $\phi'$ is irreducible.
\end{lemma} 

\begin{proof}
Let $P$ be a prime component of $\Delta$.
Let $P^\nu$ be its normalization.
If the restriction of $P$ to a general fiber of $\phi$ is not prime, then
the morphism $p\colon P^\nu\rightarrow Z$ does not have connected fibers.
We let $Z_0\rightarrow Z$ be the finite morphism associated to the Stein factorization of $p$. Then, we let $X_0$ be the normalization of the fiber product $X\times_{Z_0} Z$. 
The pull-back of $P$ on $X_0$ satisfies that the restriction of its prime components to the general fiber of $\phi_0\colon X_0\rightarrow X_0$ are irreducible. Proceeding inductively with the other prime components of $\Delta$, the result follows.
\end{proof} 

\subsection{Standard coefficients}
In this subsection, we prove a lemma
regarding coefficients of pairs admitting
int-amplified endomorphisms.
The following lemma is essentially from~\cite[Proposition 5.6]{Yos21},
we prove it for the sake of completeness.

\begin{lemma}\label{lem:stand-coeff}
Let $(X,\Delta)$ be a log Calabi--Yau pair.
Let $f\colon X\rightarrow X$ be an
int-amplified endomorphism
with $f^*(K_X+\Delta)=K_X+\Delta$.
Then, the divisor $\Delta$ must have standard coefficients.
\end{lemma}

\begin{proof}
Let $\mathcal{F}$ be the set of prime divisors $P$ of $X$ 
for which ${\rm coeff}_P(\Delta)$ is not a standard coefficient.
Since $0$ is a standard coefficient, $\mathcal{F}$ is a finite set. 

We argue that if $P\in \mathcal{F}$, then every prime component of 
$f^{-1}(P)$ also belongs to $\mathcal{F}$.
Assume by contradiction that a prime component $Q$ in $f^{-1}(P)$ is not in 
$\mathcal{F}$.
By Riemann-Hurwitz we can write 
\begin{equation}\label{eq:RH}
{\rm coeff}_Q(\Delta)=
1-(1-{\rm coeff}_P(\Delta))r_Q,
\end{equation} 
where $r_Q$ is the ramification of $f$ at $Q$.
If ${\rm coeff}_Q(\Delta)=1$, then by~\eqref{eq:RH}, we get ${\rm coeff}_P(\Delta)=1$.
On the other hand, if ${\rm coeff}_Q(\Delta)=1-\frac{1}{n}$ for some positive integer $n$, then 
from~\eqref{eq:RH}, we get 
${\rm coeff}_P(\Delta)=1-\frac{1}{r_Qn}$. We conclude that for each $P\in \mathcal{F}$
every prime component $Q$ of $f^{-1}(P)$ is in $\mathcal{F}$.
Then, for each $P \in \mathcal{F}$
and some $\ell\in \zz_{>0}$, we have 
$f^{-\ell}(P)=P$.
Let $N>1$ be the positive integer for which
$(f^{\ell})^*P = NP$ (see,~\cite[Theorem 3.3(2)]{Men17}).
Note that the previous equality holds for the smooth locus on which we can pull-back divisors.
Hence, applying~\eqref{eq:RH} to 
$f^{q\ell}$ for $q\in \zz_{>0}$, we get
\[
{\coeff}_{P}(\Delta)=
1-(1-{\rm coeff}_P(\Delta))Nq,
\]
so 
\[
{\rm coeff}_P(\Delta)=1.
\]
This contradicts the fact that $P\in \mathcal{F}$. 
Hence, $\mathcal{F}$ is empty and we conclude that $\Delta$ has standard coefficients.
\end{proof} 

\subsection{Small $\mathbb{Q}$-factorializations}

In this subsection, we prove a lemma that will allow us to take small $\qq$-factorializations in the proof of Theorem~\ref{thm:main-pdlt}.

\begin{lemma}\label{lem:passing-to-small-Q-fact}
Let $(X,\Delta)$ be a log Calabi--Yau pair with pdlt singularities.
Assume that there is a sequence of morphisms
\[
\xymatrix{
(X,\Delta)=:(X_0,\Delta_0) &(X_1,\Delta_1) \ar[l]_-{f_1} &\;\cdots\; \ar[l]_-{f_2} &(X_k,\Delta_k) \ar[l]_-{f_k}
}
\]
such that each $f_i$ is either a crepant finite morphism or a birational contraction, and $(X_k,\Delta_k)$ admits a weak toric log Calabi--Yau fibration over an abelian variety.
Assume that each $(X_i,\Delta_i)$ is pdlt.
Then, some finite cover of $(X,\Delta)$ admits a weak toric log Calabi--Yau fibration over an abelian variety.
\end{lemma} 

\begin{proof}
By using the Stein factorization, we may assume that 
$k=2$ and we have 
\[
\xymatrix{
(X_0,\Delta_0) &(X_1,\Delta_1)\ar[l]_-{f_1} &(X_2,\Delta_2) \ar[l]_-{f_2}
}
\]
where $f_1$ is finite and $f_2$ is birational.
By assumption, there is a weak toric log Calabi--Yau fibration $p_2\colon (X_2,\Delta_2)\rightarrow A$.
Note that $p_2(f_2^{-1}(x))$ is zero dimensional on $A$ for every $x\in X_1$.
Indeed, the fibers of $f_2$ are rationally chain connected~\cite[Corollary 1.6]{HM07}.
We argue that $X_1$ admits a projective fibration onto $A$.
Indeed, if $H$ is an ample divisor on $A$, then 
$p_2^*H$ is a semiample divisor on $X_2$.
Since $p_2^*H$ is trivial along the fibers of $f_2$, it descends to a semiample divisor
${f_2}_*p_2^*H$ on $X_1$.
The ample model of ${f_2}_*p_2^*H$ is $A$ so
$X_1$ admits a projective fibration $p_1\colon X_1\rightarrow A$.
Note that the general fiber of $p_1$ is the image of the general fiber of $p_2$, so it is a projective toric variety.
It then follows that 
$(X_1,\Delta_1)$ admits a weak toric log Calabi--Yau fibration to an abelian variety.
\end{proof} 

\subsection{Polarized endomorphisms and fibrations}
In this subsection, we prove a lemma regarding polarized endomorphisms and fibrations.

\begin{lemma}\label{lem:polarized-with-fibration}
Let $X$ be a normal projective variety
admitting a fibration $\phi\colon X\rightarrow Y$.
Consider a commutative diagram
\[
\xymatrix{
X\ar[d]_-{\phi}\ar[r]^-{f} & X\ar[d]^-{\phi} \\
Y\ar[r]^-{f_Y} & Y
}
\]
where both $f$ and $f_Y$ are surjective endomorphisms.
If $f$ induces an isomorphism on very general fibers of $\phi$, then $f$ is not of int-amplified type.
\end{lemma}

\begin{proof}
For the sake of contradiction, we assume that $f$ is an int-amplified endomorphism.
Let $H$ be an ample Cartier divisor on $X$ for which $f^*H-H$ is ample.
Then, we have that the sequence of volumes
$\vol({f^{n}}^*H)$ is unbounded.
In particular, for a very general point $y\in Y$ the sequence of volumes $\vol({{f^n}}^*H|_{X_y})$ is unbounded.
However, the previous sequence agrees with 
\[
\vol\left({{f|^n_{X_y}}}^*\left(H|_{X_{f_Y^n(y)}}\right)\right).
\]
Being $f^n$ an isomorphism on very general fibers, the morphism 
$f|_{X_y}^n$ is an isomorphism. So, the previous volume must be bounded above.
This leads to a contradiction.
\end{proof}

\begin{lemma}\label{lem:degenerate-divisors-pullback}
Let $X$ be a normal projective variety
admitting a fibration $\phi\colon X\rightarrow Y$.
Consider a commutative diagram
\[
\xymatrix{
X\ar[d]_-{\phi}\ar[r]^-{f} & X\ar[d]^-{\phi} \\
Y\ar[r]^-{g} & Y
}
\]
where both $f$ and $g$ are finite morphisms. If $D$ is a $\phi$-exceptional (resp. $\phi$-degenerate) divisor on $X$, then every irreducible component of $g^{-1}(\phi(D))$ is the image of a $\phi$-exceptional (resp. $\phi$-degenerate) divisor on $X$.

Moreover, if $D$ is a $\phi$-exceptional divisor such that $\phi(D)$ is dominated by $k$ prime divisors, then every irreducible component of $g^{-1}(\phi(D))$ is also dominated by $k$ prime divisors.
\end{lemma}

\begin{proof}

Suppose $D$ is a $\phi$-exceptional divisor on $X$.
Since $f$ and $g$ are finite morphisms, it is clear that every irreducible component of $g^{-1}(\phi(D))$ is the image of a $\phi$-exceptional divisor on $X$.

Now suppose $D$ is a $\phi$-degenerate divisor on $X$. 
Let $P_1,\ldots,P_k$ be the prime divisors in $X$ dominating $\phi(D),$ and let $Q_1,\ldots,Q_s$ be the irreducible components of $g^{-1}(\phi(D))$. 
For each $i=1,\ldots,s$, let $Q_{i,1},\ldots, Q_{i,j_i}$ be the prime divisors in $X$ dominating $Q_i$, where $j_i\geq 1.$
We want to show that $j_i=k$ for all $i$.
Assume without loss of generality that $j_1<k$, and that $\phi^*(Q_1)$ dominates $P_1$ but not $P_2$. 
Then, over a general point $d \in D$, we may choose a general curve $C$ in the fiber $\phi^{-1}(d)$ such that $C$ is contained in $P_1$ and intersects $P_2$ positively.
Consider the preimage $C'=f^{-1}(C)$ via the finite morphism $f$, which has a component $C_0$ contained in the prime divisor, say $Q_{1,1}$, dominating $P_1.$
By the projection formula, we have $f^*(P_2)\cdot C_0 = P_2\cdot f_*(C_0).$
Since $C_0$ maps to a general point of $Q_1$ while all the other divisors $Q_{i,j}$ for $i\neq 1$ do not map onto $Q_1,$ we have $f^*(P_2)\cdot C_0=0.$
However, this contradicts $P_2\cdot f_*(C_0)>0.$
Therefore, $Q_1$ is the image of a $\phi$-exceptional divisor, specifically, we must have $j_1=k$.
\end{proof}

\begin{lemma}\label{lem:finite-deg-div}
Let $X$ be a $\qq$-factorial variety
and let $\phi\colon X\rightarrow Z$ be a Fano type fibration.
Then $\phi$ has only finitely many degenerate divisors.
\end{lemma}

\begin{proof}
Let $k$ be the number of degenerate divisors of $X\rightarrow Z$. 
Let $E$ be a degenerate divisor over $Z$.
Then, the $E$-MMP over $Z$ terminates after contracting $E$ (see~\cite[Lemma 2.10]{Lai11}). 
Let $X\dashrightarrow Y$ be such MMP over $Z$.
The number of degenerate divisors of $Y\rightarrow Z$ is at least $k-2$
and $\rho(Y/Z)=\rho(X/Z)-1$.
We conclude that $2k\leq \rho(X/Z)$,
so $k$ is finite.
\end{proof} 

The aim of the following two lemmas is to prove that an int-amplified endomorphism $f\colon A\to A$, with $A$ an abelian variety, fixes no subvariety of $A$.

\begin{lemma}\label{lem:pol-no-fixed}
Let $k$ be a positive integer.
Let $A$ be an abelian variety
and 
$D$ be an ample Cartier divisor on $A$.
Let $f\colon A\rightarrow A$ be an endomorphism
and let 
$Z\subsetneq A$ be a proper subvariety.
Assume that for every $n$
the subvariety $Z_n:=f^{-n}(Z)$ is irreducible and the volume 
$\vol(D|_{Z_n})$ is bounded above by $k$.
Then, the endomorphism $f$ is not polarized.
\end{lemma}

\begin{proof}
Assume otherwise that $f$ is a polarized endomorphism.
Let $H$ be an ample Cartier divisor on $A$ such that $f^*H\equiv qH$ with $q>1$.
Then, we have that $\deg(f^\ell)=q^{\ell\dim A}$ for every $\ell\geq 1$.
Let $f_{n,m}\colon Z_n\rightarrow Z_m$ be the restriction of $f^{n-m}$ to $Z_n$
for every $n>m$.
Since $f$ is \'etale, we have that 
$\deg(f_{n,m})=q^{\dim(X)(n-m)}$
for every $n>m$.
Since $\vol(D|_{Z_n})\leq k$ for every $n$, there exists a positive integer
$k'$ such that
$\vol(H|_{Z_n})\leq k'$ for every $n$.
In particular, we can find $n>m$ such that $\vol(H|_{Z_n})=\vol(H|_{Z_m})$.
Note that $(f^{n-m})^*H=q^{n-m}H$.
Hence, we have 
$f_{n,m}^*H|_{Z_m}=q^{n-m}H|_{Z_n}$.
Thus, by taking degree we get
\[
\deg(f_{n,m})\vol(H|_{Z_m})=
q^{\dim(Z)(n-m)}\vol(H|_{Z_n}),
\]
from where we deduce that 
$\deg(f_{n,m})=q^{\dim(Z)(n-m)}=q^{\dim(X)(n-m)}$.
This leads to a contradiction
as $\dim(Z)<\dim(X)$.
\end{proof}

\begin{lemma}\label{lem:int-amplified-abelian-no-fixed-sub}
Let $k$ be a positive integer.
Let $A$ be an abelian variety
and $D$ be an ample Cartier divisor on $A$.
Let $f\colon A\rightarrow A$ be an endomorphism and let $Z\subsetneq A$ be a proper subvariety.
Assume that for every $n\geq 1$
the subvariety $Z_n:=f^{-n}(Z)$ is irreducible and the volume $\vol(D|_{Z_n})$ is bounded above by $k$.
Then, the endomorphism $f$ is not int-amplified.
\end{lemma} 

\begin{proof}
The statement holds in dimension one, so we proceed by induction on the dimension.

For the sake of contradiction 
assume that $f$ is int-amplified.
Let $g\colon A'\rightarrow A$ be an isogeny such that $A'$ is a product of simple abelian varieties.
By~\cite[Theorem 3.1]{MYY24}, up to replacing $f$ with an iteration and $A'$ with a finite cover of $A'$, we have a commutative diagram:
\[
\xymatrix{
A'\ar[d]_-{g}\ar[r]^-{f'} & A'\ar[d]^-{g} \\ 
A\ar[r]^-{f} & A \\
}
\]
where $f'$ is an int-amplified endomorphism.
A priori the subvarieties $W_n:=g^{-1}(Z_n)$ may not be irreducible.
However, the number of irreducible components of $W_n$ is bounded above by $\deg g$.
On the other hand, we have $\vol(g^*D|_{W_n})\leq \deg g k$.
Thus, we may replace $A$ with $A'$ and each $Z_n$ with a suitable irreducible component of $W_n$.
By doing so, we may assume that $A$ is itself a product of simple abelian varieties.

From now on, we assume that $A$ is a product of simple abelian varieties.
The endomorphism $f$ fixes the ample cone of $A$.
Hence, there exists a pseudo-efffective divisor $E$ such that $f^*E \equiv qE$ for some $q\in \qq$.
By~\cite[Page 802, Proposition 0.5]{Mor15}, we may assume that $E$ is indeed an effective divisor.
In particular, the divisor $E$ is semiample. 
By~\cite[Theorem 3.1]{Men17}, we know that $q>1$.
Let $\phi \colon A\rightarrow A_0$ be the ample model of $E$. In particular, we know that $A\rightarrow A_0$ is the projection onto the product of some of the simple abelian components of $A$.
We have that $E\equiv \phi^*E_0$ for some ample divisor $E_0$ on $A_0$.
Then, we get a commutative 
\[
\xymatrix{
A\ar[d]_-{\phi} \ar[r]^-{f} & A\ar[d]^-{\phi} \\
A_0\ar[r]^-{f_0} & A_0
}
\]
where $f_0$ is a polarized endomorphism.
Indeed, we get $f_0^* E_0\equiv qE_0$.
If $Z_n$ is vertical over $A_0$ for infinitely many $n$, then
we get a contradiction by Lemma~\ref{lem:pol-no-fixed}.
If $Z_n$ is horizontal over $A_0$ for infinitely many $n$, then $\dim(A_0)<\dim A$.
Let $F$ be the general fiber of $\phi$.
In this case, the number of components of 
$Z_n\cap F$ and $\vol(D|_{Z_n\cap F})$ are bounded above, independently of $n$.
Therefore, we get a contradiction by induction on the dimension by restricting to $F$.
Thus, we conclude that $f$ is not int-amplified.
\end{proof} 

\section{Bounded finite quotients}
\label{sec:bounded-finite-quot}

In this section, we study sequences of finite quotients $X\rightarrow X/G_i$ such that the images belong to a bounded family.
The boundedness allows us to deduce information about the groups $G_i$ as well as the variety $X$.
The results in this section are similar to those in~\cite{MYY24}. However, in those results, we assume that all quotients were isomorphic. 

\begin{definition}
{\em We say that a family of pairs $\mathcal{F}$ is {\em log bounded} if there exists a scheme of finite type $T$ over the base field satisfying the following condition: There is a projective morphism
$\pi\colon \mathcal{X} \rightarrow T$ with a divisor $\mathcal{B}$ on $\mathcal{X}$ such that for every $(X,B)\in \mathcal{F}$ we have $(X,B)\simeq (\mathcal{X}_t,\mathcal{B}_t)$ for a suitable $t\in T$.

We say that a family of varieties $\mathcal{F}$ is {\em bounded} if the family 
$\{ (X,0)\mid X\in \mathcal{F} \}$ is a log bounded family of pairs.
}
\end{definition}

\begin{lemma}\label{lem:finite-covers}
Let $\mathcal{F}$ be a log bounded family of log Calabi--Yau pairs $(X,\Delta)$ 
with $K_X+\Delta\sim 0$
and $d$ be a positive integer.
Let $\mathcal{G}$ be the family of log pairs 
$(Y,\Delta_Y)$, with $\Delta_Y\geq 0$, for which there exists a finite morphism $f\colon Y\rightarrow X$ of degree at most $d$ with $f^*(K_X+\Delta)=K_Y+\Delta_Y$ for some $(X,\Delta)\in \mathcal{F}$.
Then, the family $\mathcal{G}$ is log bounded.
\end{lemma} 

\begin{proof}
Since $\mathcal{F}$ is a log bounded family,
for each element $(X,\Delta) \in \mathcal{F}$, we may find a reduced ample effective divisor $H$ with $\vol(H)\leq v(\mathcal{F})$, 
where $v(\mathcal{F})$ is a positive number only depending on $\mathcal{F}$,
such that $(X,\Delta+H)$ is log canonical.
In particular, we get that $(Y,\Delta_Y+f^*H)$ is log canonical for every log pair $(Y,\Delta_Y)\in \mathcal{G}$.
Note that $K_Y+\Delta_Y+f^*H$ is ample
and $\vol(K_Y+\Delta+f^*H)=\vol(f^*H)\leq d^{m(\mathcal{F})}\vol(H)$ where $m(\mathcal{F})$ is a positive integer that only depends on $\mathcal{F}$.
By~\cite[Theorem 1.2.1]{HMX18}, we conclude that the log pairs $(Y,\Delta_Y+f^*H)$ belong to a log bounded family.
In particular, the family $\mathcal{G}$ is log bounded.
\end{proof} 

The proof of the following lemma is similar to that of~\cite[Lemma 4.1]{MYY24}. 

\begin{lemma}\label{lem:bounded-quot-non-triv-boundary}
Let $\mathcal{F}$ be a log bounded family of log Calabi--Yau pairs.
Let $(X,\Delta)$ be a log Calabi--Yau pair with $K_X+\Delta \sim 0$. 
Let $\{G_i\}_{i\in \nn}$ be an infinite nested sequence of finite subgroups with $G_i\leqslant \mathbb{T}\leqslant {\rm Aut}^0(X,\Delta)$
for each $i$.
Here, $\mathbb{T}$ denotes an algebraic torus.
If $(X/G_i,\Delta/G_i)\in \mathcal{F}$ holds for every $i\in \nn$, then 
$\Delta \neq 0$.
\end{lemma} 

\begin{proof}
Let $n$ be the dimension
of the variety $X$.
Assume that the statement does not hold, meaning that $X/G_i \in \mathcal{F}$ for every $i$ and $\Delta=0$.
Since the sequence of varieties $X/G_i$ belongs to a bounded family, there exists a constant $m_0$ such that $\mu(X/G_i;x)\leq m_0$ for every closed point $x\in X/G_i$.
Let $\pi\colon \widetilde{X}\rightarrow X$ 
be a $\mathbb{T}$-equivariant projective birational morphism such that $\widetilde{X}$ admits a quotient for the $\mathbb{T}$-action. Thus, there is a $\mathbb{T}$-equivariant fibration 
$\widetilde{X}\rightarrow Z$
which is a quotient for the torus action.
Let $\pi^*(K_X)=K_{\widetilde{X}}+
\widetilde{\Delta} \sim 0$.
The sub-pair $(\widetilde{X},\widetilde{\Delta})$
is sub-log Calabi--Yau
and its restriction to a general fiber of $\widetilde{X}\rightarrow Z$ is a toric log Calabi--Yau pair.
Then, there is a prime component $S$ of $\widetilde{\Delta}^{=1}$ that dominates $Z$.
Let $\mathbb{G}_m=:\mathbb{T}_0\leqslant \mathbb{T}$ be the one-dimensional torus that acts as the identity on $S$.
Let $\mathbb{T}_1$ be a split torus of $\mathbb{T}_0$ in $\mathbb{T}$.
For each $i$, we denote by $H_i$ the intersection of $G_i$ with $\mathbb{T}_0$.
There are two cases; either the order of $H_i$ diverges or is bounded above by a constant which is independent of $i$.

First, assume that the order of $H_i$ diverges.
Let $\pi(S)\subset X$.
By~\cite[Theorem 10.1]{AH06}, the variety $\pi(S)$ admits a subvariety $S_0$ on which $\mathbb{T}_1$ acts faithfully.
Let $x\in S_0$ be a general point.
Then, the pair $(X;x)$ is a normal singularity with a $\mathbb{T}_0$-action. Let $x_i$ be the image of $x$ in $X/G_i$
and $y_i$ be the image of $x$ in $X/H_i$. 
Then, we have
$(X/G_i;x_i)\simeq (X/H_i;y_i)$.
By Lemma~\cite[Lemma 2.8]{MYY24}, we know that $\lim_{i\rightarrow \infty} \mu(X/H_i;y_i)=\infty$.
This contradicts the fact that $\mu(X/G_i;x_i)\leq m_0$.
We conclude that $\Delta \neq 0$.

Now, assume that the order of $H_i$ is bounded above by $d$. Let $K_i$ be the intersection of $G_i$ with $\mathbb{T}_1$. 
Then, the groups $\{K_i\}_{\in \nn}$ form an infinite nested sequence of finite groups with $K_i\leqslant \mathbb{T}_1 \leqslant {\rm Aut}^0(X,\Delta)$. 
The quotients $(X/K_i,\Delta/K_i)\rightarrow (X/G_i,\Delta/G_i)$ have degree at most $d$, so the pairs $(X/K_i,\Delta/K_i)$ are still in a bounded family $\mathcal{F}'$ that only depends on $\mathcal{F}$ and $d$ (see Lemma~\ref{lem:finite-covers}).
We replace $\{G_i\}_{i\in \nn}, \mathbb{T}$, and $\mathcal{F}$ with $\{K_i\}_{i \in \nn}, \mathbb{T}_1$, and $\mathcal{F}'$,
respectively.
By doing so, the rank of $\mathbb{T}$ decreases by one. Thus, this replacement can only happen finitely many cases. Hence, eventually, we are in the situation of the previous paragraph.
\end{proof} 

\begin{theorem}\label{thm:toric-bounded-quot}
Let $\mathcal{F}$ be a log bounded family of log Calabi--Yau pairs. 
Let $(X,\Delta)$ be a $n$-dimensional log Calabi--Yau pair with $K_X+\Delta\sim 0$.
Let $\{G_i\}_{i\in \nn}$ be an infinite nested sequence of finite subgroups of $\mathbb{T}\leqslant {\rm Aut}(X,\Delta)$.
Here, $\mathbb{T}$ is an algebraic torus.
Assume the following conditions hold:
\begin{enumerate}
\item For each log canonical center $W$ of $(X,\Delta)$, the restriction of $\{G_i\}_{i\in \nn}$ to $W$ gives an infinite nested sequence of finite actions, and 
\item If $(W^\nu,\Delta_{W^\nu})$ is the pair obtained by adjunction of $(X,\Delta)$ to the normalization $W^\nu$ of $W$, then 
\[
(W^\nu/G_i,\Delta_{W^\nu}/G_i)\in \mathcal{F}
\]
for each $i\in \nn$.
\end{enumerate} 
Then, the following conditions are satisfied:
\begin{enumerate}
\item[(i)] the algebraic torus $\mathbb{T}$ has rank $n$, 
\item[(ii)] for $i$ large enough, the group $G_i$ has rank $n$, and
\item[(iii)] the log Calabi--Yau pair $(X,\Delta)$ is a toric log Calabi--Yau pair.
\end{enumerate} 
\end{theorem} 

\begin{proof}
The statement is clear in dimension one, 
so we can proceed by induction on the dimension.

First, we reduce to the case in which $G_i\cap \mathbb{T}_0$ is an infinite nested sequence of finite groups for every subtorus $\mathbb{T}_0\leqslant \mathbb{T}$.
Assume, otherwise, that $K_i:=G_i \cap \mathbb{T}_0$ stabilizes for some subtorus $\mathbb{T}_0$ of $\mathbb{T}$.
Let $\mathbb{T}_1$ be a complementary torus of $\mathbb{T}_0$ in $\mathbb{T}$.
Then, we know that $Q_i:=G_i\cap \mathbb{T}_1$ is an infinite nested sequence of finite subgroups. 
On the other hand, there is a constant $k_0$ for which $|K_i|\leq k_0$ for every $i\in \nn$.
Hence, the quotients $(X/Q_i,\Delta/Q_i)\rightarrow (X/G_i,\Delta/G_i)$ have degree at most $k_0$.
By Lemma~\ref{lem:finite-covers}, we conclude that $(W^\vee/Q_i,\Delta_{W^\vee}/Q_i)$ belongs to a log bounded family $\mathcal{G}$ for each log canonical center $W$ of $(X,\Delta)$ and $i\in \nn$.
Moreover, the restriction of $\{Q_i\}_{i\in \nn}$ to $W$ gives an infinite nested sequence of finite actions for every log canonical center $W$ of $(X,\Delta)$.
Therefore, we may replace 
$\{G_i\}_{i\in \nn}, \mathbb{T}$, and $\mathcal{F}$ 
with 
$\{Q_i\}_{i\in \nn}, \mathbb{T}_1$, and
$\mathcal{G}$, respectively.
By doing so, the rank of the algebraic torus $\mathbb{T}$ decreases by at least one, 
so we can repeat this process only finitely many times.
Thus, after finitely many replacements, we may assume that $G_i\cap \mathbb{T}_0$ is an infinite nested sequence of finite groups for every subtorus $\mathbb{T}_0\leqslant \mathbb{T}$.

By Condition (2) applied to $X$ itself and Lemma~\ref{lem:bounded-quot-non-triv-boundary}, we know that $\Delta \neq 0$.
Let $S$ be the normalization of a prime component of $\Delta$.
Let $(S,B_S)$ be the log Calabi--Yau pair induced on $S$ by adjunction of $(X,\Delta)$.
Note that $S$ has dimension $n-1$ and $K_S+\Delta_S\sim 0$.
For each $i$, we denote by $H_i$ the restriction of $G_i$ to $S$.
As $G_i\leqslant \mathbb{T}$ for each $i$, we get that $H_i\leqslant \mathbb{T}_S$ for each $i$, where $\mathbb{T}_S$ is the restriction of $\mathbb{T}$ to $S$.
Furthermore, by condition (1) in the theorem, the groups $\{H_i\}_{i\in \nn}$ give an infinite nested sequence of finite actions on $(S,\Delta_S)$.
Moreover, by condition (2) in the theorem, the quotient by $H_i$ of any log canonical center of $(S,\Delta_S)$ belongs to $\mathcal{F}$.
Then, by induction on the dimension, we conclude that the algebraic torus $\mathbb{T}_S$ has dimension $n-1$, for $i$ large enough the group $H_i$ has rank $n-1$, and 
the log Calabi--Yau pair $(S,\Delta_S)$ is a toric log Calabi--Yau pair.
We conclude that $\mathbb{T}$ has rank either $n$ or $n-1$.
It suffices to show that the algebraic torus $\mathbb{T}$ has rank $n$.
Indeed, if $\mathbb{T}$ has rank $n$, then $(X,\Delta)$ is a toric log Calabi--Yau pair.

Assume, by the sake of contradiction, that $\mathbb{T}$ has rank $n-1$.
Let $\pi\colon \widetilde{X}\rightarrow X$ be a $\mathbb{T}$-equivariant projective birational morphism 
for which
the variety $\widetilde{X}$ admits a quotient for the torus action.
Let $\widetilde{X}\rightarrow Z$ be the $\mathbb{T}$-quotient.
Let $\pi^*(K_X+\Delta)=K_{\widetilde{X}}+\widetilde{\Delta}$. The restriction of $(\widetilde{X},\widetilde{\Delta})$ to a general fiber of $F$ is a toric log Calabi--Yau pair.
Let $S$ be a prime component of $\widetilde{\Delta}^{=1}$ that dominates $F$.
The restriction $\mathbb{T}_S$ of $\mathbb{T}$ to $S$ has rank $n-2$. Therefore, the image $\pi(S)$ of $S$ in $X$ has codimension $2$.
Otherwise, we would get a contradiction by induction on the dimension. 
Let $\mathbb{T}_0$ be the kernel 
of the surjection $\mathbb{T}\rightarrow \mathbb{T}_S$.
By the first paragraph, we know that 
$\mu_i:=G_i\cap \mathbb{T}_0$ gives an infinite sequence of finite subgroups.
The torus $\mathbb{T}_S$ acts faithfully on $\pi(S)$.
Let $x\in \pi(S)$ be a general point.
Let $X_i:=X/G_i$ and $x_i$ be the image of $x$ in the quotient.
We denote by $y_i$ the image of $x_i$ in the quotient of $X$ by $\mu_i$.
Then, we have an equality of multiplicities
\[
\mu(X/G_i;x_i)=\mu(X/\mu_i;y_i).
\]
However, by Lemma~\cite[Lemma 2.8]{MYY24}, we know that the right side diverges.
This contradicts the fact that $X/G_i$ belongs to a bounded family $\mathcal{F}$ (see, e.g.,~\cite[Lemma 2.7]{MYY24}).
This leads to a contradiction. 
We conclude that the rank of $\mathbb{T}$ equals the dimension of $X$.
\end{proof} 

\begin{lemma}\label{lem:bounded-family-sugroup-in-torus}
Let $\mathcal{F}$ be a log bounded family of log Calabi--Yau pairs.
Assume that for each $(X,\Delta)\in \mathcal{F}$,
the automorphism group ${\rm Aut}(X,\Delta)$ is a finite extension of an algebraic torus.
Let $\{(X_i,\Delta_i)\}_{i\in \zz}$ be a countable sequence of pairs in $\mathcal{F}$
and let
$f_{i,j}\colon (X_i,\Delta_i)\rightarrow (X_j,\Delta_j)$
be 
compatible\footnote{In the sense that $f_{j,k}\circ f_{i,j}=f_{i,k}$ for every $k\geq j \geq i$.} crepant finite Galois morphisms
for each $i\geq j$.
Let $G_{i,j}\leqslant {\rm Aut}(X_i,\Delta_i)$ be the Galois group corresponding to $f_{i,j}$.
Then, there exists $i_0 \in \zz$ such that
$G_{j,k}\leqslant {\rm Aut}^0(X_j,\Delta_j)$ for every $k \geq j \geq i_0$.
\end{lemma} 

\begin{proof}
For each $i\leq j\leq k$, we have a short exact sequence
\[
1\rightarrow G_{i,j} \rightarrow G_{i,k}
\rightarrow G_{j,k}\rightarrow 1.
\]
Fix $\ell_0\in \zz$.
Then, there exists $\ell_1$ such that 
for every $j\geq \ell_1$ the image of
the groups $G_{\ell_0,j}$ stabilize in 
${\rm Aut}(X_{\ell_0},\Delta_{\ell_0})/{\rm Aut}^0(X_{\ell_0},\Delta_{\ell_0})$. Hence, we conclude that $G_{\ell_0,j}/G_{\ell_0,\ell_1}\simeq G_{\ell_1,j}$ is abelian for every $j\geq \ell_1$.

Let $Z_j$ be the centralizer of $G_{\ell_1,j}$ in ${\rm Aut}(X_{\ell_1},\Delta_{\ell_1})$.
As $G_{\ell_1,j}$ are abelian groups, they are contained in their centralizer $Z_j$.
The group $Z_j$ stabilizes for $j$ large enough.
Indeed, this follows from the fact that 
$Z_j\cap \mathbb{T}_{\ell_1}$ stabilize for $j$ large enough. Here, $\mathbb{T}_{\ell_1}$ is the connected component of the automorphism group
of $(X_{\ell_1},\Delta_{\ell_1})$.
Let $Z_\infty$ be the limit of the centralizers.
Note that $Z_\infty$ is an algebraic group
that contains all the groups $G_{\ell_1,j}$.

Let $Z_\infty^0$ be the connected component of $Z_\infty$.
Then, there exists $i_1\in \zz$ such that
the image of $G_{\ell_1,j}$ stabilize in 
$F_\infty:=Z_\infty/Z_\infty^0$ for every $j\geq i_1$.
We denote by $\pi_\infty\colon Z_\infty\rightarrow F_\infty$ the projection onto the groups of components.
Let $j \geq i_1$.
Let $N_j$ be the normalizer of $G_{\ell_1,j}$ in ${\rm Aut}(X_{\ell_1},\Delta_{\ell_1})$.
Then, we have a natural homomorphism 
$N_j/G_{\ell_1,j} \rightarrow {\rm Aut}(X_j,\Delta_j)$ for every $j$.
Note that we have a commutative diagram:
\begin{equation}\label{diag1}
\xymatrix{
Z^0_\infty/(Z^0_\infty\cap G_{\ell_1,j})\ar[d] 
\ar@{^{(}->}[r] & Z_\infty/G_{\ell,j} \ar@{^{(}->}[r] &
N_j/G_{\ell_1,j} \ar[d] \\ 
{\rm Aut}^0(X_j,\Delta_j) \ar@{^{(}->}[rr] & &
{\rm Aut}(X_j,\Delta_j).
}
\end{equation}
On the other hand, for every $k\geq j$, there is a commutative diagram
\begin{equation} 
\label{diag2}
\xymatrix{
1 \ar[r] & Z_\infty^0/(Z_\infty^0\cap G_{\ell_1,j}) \ar[r] & 
Z_\infty/G_{\ell_1,j} \ar[r] & F_\infty/\pi_\infty(G_{\ell_1,j}) \ar[r] & 1 \\
1 \ar[r] & (Z_\infty^0\cap G_{\ell_1,k}) /(Z_\infty^0\cap G_{\ell_1,j}) \ar[u] \ar[r] & G_{\ell_1,k}/G_{\ell_1,j} \ar[u] \ar[r] & \pi_\infty(G_{\ell_1,k})/\pi_\infty(G_{\ell_1,j}) \ar[r] \ar[u] & 1 \\
}
\end{equation}
By the commutative diagram~\eqref{diag2}, and the fact that $G_{\ell_1,k}$ and $G_{\ell_1,j}$ have the same image in $F_\infty$, we conclude that 
the image of $G_{\ell_1,k}/G_{\ell_1,j}$ in $Z_\infty/G_{\ell_1,j}$ lies in 
$Z^0_\infty/(Z_\infty^0\cap G_{\ell_1,j})$.
By the commutative diagram~\eqref{diag1}, we conclude that the image of $G_{\ell_1,k}/G_{\ell_1,j}$ lies in the connected component of ${\rm Aut}(X_j,\Delta_j)$.
Thus, for every $k\geq j\geq \ell_1$, the group $G_{j,k}$ is contained in ${\rm Aut}^0(X_j,\Delta_j)$ for every $k\geq j\geq \ell_1$.
\end{proof} 

\section{Lifting int-amplified endomorphisms to small $\qq$-factorializations}

In this section, we prove that we can lift
a surjective endomorphism
to a small $\qq$-factorialization.

\begin{theorem}\label{thm:lift-to-Q-fact}
Let $X$ be a klt type variety, and let $(X,\Delta)$ be a log Calabi--Yau pair.
Let $f\colon (X,\Delta) \to (X,\Delta)$ be
a surjective endomorphism.
Then, up to replacing $f$ with an iteration, there exists
a small $\qq$-factorialization $\phi\colon (Y,\Delta_Y) \to (X,\Delta)$
such that $(Y,\Delta_Y)$ admits an int-amplified endomorphism $f_Y$
making the following diagram commute
\[
    \xymatrix{
        (Y,\Delta_Y) \ar[r]^{f_Y} \ar[d]_{\phi} & (Y,\Delta_Y) \ar[d]^{\phi}\\
        (X,\Delta) \ar[r]_{f} & (X,\Delta)
    }
\]
\end{theorem}

\begin{proof}
As the variety $X$ is of klt type, it admits a small 
$\qq$-factorialization $\phi_X \colon (X',\Delta_{X'}) \to (X,\Delta)$. 
The morphism $\phi_X$ is of Fano type, so in particular, 
$X' \xrightarrow{\phi} X$ is a relative Mori Dream Space 
(see, e.g., \cite[Theorem 3.18]{BM21}).
\bigskip

\emph{Step 1.}

Consider the diagram
\[
    \xymatrix{
        (X_n,\Delta_{X_n}) \ar[r]^{g_n} \ar[d]_{\phi_n} & (X',\Delta_{X'}) \ar[d]^{\phi_X}\\
        (X,\Delta) \ar[r]_{f^n} & (X,\Delta),
    }
\]
where $(X_n,\Delta_{X_n})$ is the main component of the normalization of the fiber product.
The map $\phi_X$ is small, which implies that $\phi_n$ is small,
$X_n$ is of klt type,
but $X_n$ might no be $\qq$-factorial.
For each $n$, take a small $\qq$-factorialization 
$(X'_n,\Delta_{X'_n}) \to (X_n,\Delta_{X_n})$,
which in turn gives a small $\qq$-factorialization
\[(X'_n,\Delta_{X'_n}) \to (X_n,\Delta_{X_n}) \to (X,\Delta).\]
The morphism $X' \to X$ is a relative Mori Dream Space, 
so $(X',\Delta')$ has finitely many small $\qq$-factorial modifications over $X$.
In particular,  there exists $n \in \nn$ such that for infinitely many $m\in \nn$,
$(X'_n,\Delta_{X'_n})$ is isomorphic to $(X'_m,\Delta_{X'_m})$.

Replace $(X',\Delta_{X'})$ with $(X_n,\Delta_{X_n})$, 
and only consider the pairs $(X_m, \Delta_{X_m})$ such that
$(X'_m,\Delta_{X'_m})$ is isomorphic to $(X'_n,\Delta_{X'_n})$
\[
    \xymatrix{
        (X'_n,\Delta_{X'_n})\ar[d] & (X'_n,\Delta_{X'_n})\ar[d]\\
        (X_m,\Delta_{X_m}) \ar[r]^{g_{k_m}} \ar[d]_{\phi_m} & (X_n,\Delta_{X_n}) \ar[d]^{\phi_n}\\
        (X,\Delta) \ar[r]_{f^{k_m}} & (X,\Delta).
    }
\]
Once again, as the map $X'_n \to X$ is a relative Mori Dream Space,
it has finitely many intermediate contractions, so for $m_1, m_2 \gg 0$,
$(X_{m_1}, \Delta_{X_{m_1}})$ is isomorphic to $(X_{m_2}, \Delta_{X_{m_2}})$,
and we obtain, up to replacing $f$ with an iteration,
\[
    \xymatrix{
        (X_{m_1},\Delta_{X_{m_1}}) \ar[d]_{\phi_{m_1}} \ar[r]^{g} & (X_{m_1},\Delta_{X_{m_1}}) \ar[d]^{\phi_{m_1}} \\
        (X,\Delta) \ar[r]_{f} & (X,\Delta).
    }
\]
Let $(Y_1,\Delta_1):= (X_{m_1},\Delta_{X_{m_1}})$ and $\phi_1 = \phi_{m_1}$.
\bigskip

\emph{Step 2.}

We may repeat the argument of Step 1 with the pair $(Y_1,\Delta_1)$ 
to obtain a small morphism $(Y_2,\Delta_2) \xrightarrow{\phi_2} (Y_1,\Delta_1)$ 
such that we can lift an iteration of $f$ to $(Y_2,\Delta_2),$
and in general,
we may find a small morphism $(Y_{k+1},\Delta_{k+1}) \xrightarrow{\phi_{k+1}} (Y_k,\Delta_k)$,
for all $k\in \nn$,
with the same properties as above.
We want to show that $Y_k$ is $\qq$-factorial, for some $k\in \nn$.

For each $k\in \nn$, let $(Y'_k,\Delta'_k) \xrightarrow{\psi_k} (Y_k,\Delta_k)$ 
be a small $\qq$-factorialization. 
As argued in Step 1, there are finitely many $(Y'_k,\Delta'_k)$,
so we may consider the case where all $(Y'_k,\Delta'_k)$
are isomorphic to a fixed $(Y',\Delta'_Y)$.
Then, as above, $Y' \to X$ has finitely many intermediate contractions.
Take $k \in \nn$ larger than the total number of possible 
intermediate contractions of $Y' \to X$.
Then one of the maps
\[
    Y' \xrightarrow{\psi_k} Y_k \xrightarrow{\phi_k} Y_{k-1} \to \cdots \to Y_1 \xrightarrow{\phi_1} X
\]
must be an isomorphism. 
If $\psi_k$ is an isomorphism, then $Y_k$ is $\qq$-factorial. 
If $\phi_j$ is an isomorphism for some $j$, 
then it means that the starting $\qq$-factorialization $Y_j \xrightarrow{\phi_{Y_{j-1}}} Y_{j-1}$
was an isomorphism, so $Y_{j-1}$ is $\qq$-factorial.

Therefore, there exists a small $\qq$-factorialization $\phi\colon (Y,\Delta_Y) \to (X,\Delta)$ a surjective endomorphism $f_Y \colon (Y,\Delta_Y) \to (Y,\Delta_Y)$ such that the diagram
\[
    \xymatrix{
        (Y,\Delta_Y) \ar[r]^{f_Y} \ar[d]_{\phi} & (Y,\Delta_Y) \ar[d]^{\phi}\\
        (X,\Delta) \ar[r]_{f} & (X,\Delta)
    }
\]
commutes.
\end{proof}

\begin{remark}
    By \cite[Theorem 3.3]{Men17}, if $f$ is int-amplified, then $f_Y$ is int-amplified.
\label{rem:lift_is_int}
\end{remark}

\begin{remark}
    The proof of Theorem \ref{thm:lift-to-Q-fact} still holds for a surjective endormorphism $f\colon X \to X$, but one has to drop the Calabi--Yau pair structure. 
\end{remark}

\section{Int-amplified endomorphisms on pdlt pairs}

In this section, we prove the following theorem
regarding int-amplified endomorphisms on pdlt pairs.

\begin{theorem}\label{thm:main-pdlt}
Let $(X,\Delta)$ be a log Calabi--Yau pair with pdlt singularities.
Assume that $(X,\Delta)$ admits an int-amplified endomorphism.
Then $(X,\Delta)$ is a finite quotient of a toric log Calabi--Yau fibration over an abelian variety.
\end{theorem}

\begin{proof}
First, we aim to prove that a finite cover of $(X,\Delta)$ admits a weak toric log Calabi--Yau fibration to an abelian variety.

By Lemma~\ref{lem:stand-coeff}, we know that $\Delta$ has standard coefficients.
Thus, we can take the index one cover of $K_X+\Delta\sim_\qq 0$, call it $p \colon Y\rightarrow X$, and by Riemann-Hurwitz, we get 
$K_Y+\Delta_Y=p^*(K_X+\Delta)$ where 
$\Delta_Y$ is a reduced divisor and $K_Y+\Delta_Y\sim 0$.
By~\cite[Corollary 3.2]{MYY24}, the pair
$(Y,\Delta_Y)$ admits an int-amplified endomorphism $f_Y$.
By Lemma~\ref{lem:finite-cover-pdlt}, we know that $(Y,\Delta_Y)$ has pdlt singularities.
The pair $(Y,\Delta_Y)$ admits a small $\qq$-factorialization as it is of klt type.
Further, the small $\qq$-factorialization is still pdlt.
Then, by Lemma~\ref{lem:passing-to-small-Q-fact}, Theorem \ref{thm:lift-to-Q-fact} and Remark \ref{rem:lift_is_int}, we may assume that $Y$ is $\qq$-factorial.

Now, we can apply~\cite[Theorem 6.6]{Yos21}, to obtain a commutative diagram as follows:
\[
\xymatrix{
(Y,\Delta_Y) \ar[r]^{f_Y} \ar[d]_{\phi_Y} &(Y,\Delta_Y) \ar[d]^{\phi_Y} \\
A \ar[r]^{f_A} &A,
}
\]
where the following conditions are satisfied:
\begin{enumerate}
\item the variety $Y$ is $\qq$-factorial,
\item the pair $(Y,\Delta_Y)$ is pdlt, 
\item $K_Y+\Delta_Y\sim 0$,
\item $\phi_Y$ is of Fano type, and
\item $A$ is an abelian variety.
\end{enumerate}
Let $F$ be a general fiber of $\phi_Y$
and $\Delta_F$ be the restriction of $\Delta_Y$ to $F$.
By~\cite[Lemma 1.5.C]{Nor83} (see also \cite[Lemma 3.13]{FM23}), we have an exact sequence of fundamental groups:
\[
\xymatrix{ 
\pi_1^{\rm reg}(F,\Delta_F) 
\ar[r]^{\psi} 
&\pi_1^{\rm reg}(Y,\Delta_Y)\ar[r]
&\pi_1(A) \ar[r] &1. 
}
\]
The group $\pi_1^{\rm reg}(F,\Delta_F)$ is finitely generated, 
and by~\cite[Theorem 3.3]{MYY24},
we know that it is virtually abelian.
In particular, it is a residually finite group. 
The image of $\psi$ is again finitely generated and virtually abelian,
so it is residually finite.
Since the extension of a finitely generated residually finite group 
by a cyclic group
is again residually finite, we conclude,
by an induction argument,
that $\pi_1^{\rm reg}(Y,\Delta_Y)$ is residually finite.
The image of $\psi$ is virtually abelian.
As $\pi_1^{\rm reg}(Y,\Delta_Y)$ is residually finite it admits a normal subgroup of finite index 
whose intersection with ${\rm Im}(\psi)$ is abelian.
Hence, passing to a finite cover of $(Y,\Delta_Y)$, we may assume that the image of $\psi$ is indeed abelian.
By~\cite[Theorem 3.1]{MYY24}, we may assume that the finite cover admits an int-amplified endomorphism.
Thus, for every general closed point $a\in A$, the finite morphism
\[
f_Y\colon (Y_a,\Delta_a) \rightarrow 
(Y_{f_A(a)},\Delta_{f_A(a)})
\]
is Galois.
Note that, by abuse of notation, we are
denoting by $f_Y$ the restriction to the fiber.
Let $U$ be an open set for which
$Y_u$ is Fano type and klt 
and $(Y_u,\Delta_u)$ is log canonical 
for every $u\in U$.
Let $Z$ be its complement. 
We consider the set $V:=A\setminus \cup_{i\in \zz} f_A^i(Z)$.
Let $v\in V$
and let $G_i \leqslant {\rm Aut}(Y_v,\Delta_v)$ 
be the Galois group corresponding to the finite morphism
\[
f_Y^i \colon (Y_v,\Delta_v) 
\rightarrow \left(Y_{f_A^i(v)},\Delta_{f_A^i(v)}\right).
\]
By~\cite[Theorem 2.10]{MYY24} and Lemma~\ref{lem:bounded-family-sugroup-in-torus}, we conclude that $G_i\leqslant {\rm Aut}^0(Y_v,\Delta_v)$ for all $i$ large enough. 
The restriction of $G_i$ to every stratum of $\Delta_v$ is non-trivial. 
Otherwise, by applying~\cite[Lemma 2.6]{MYY24} and Lemma~\ref{lem:polarized-with-fibration}, we get a contradiction.
Then, we can apply Theorem~\ref{thm:toric-bounded-quot} to conclude that $(Y_v,\Delta_v)$ is a toric log Calabi--Yau pair. 
We have concluded that the very general fiber of $(Y,\Delta_Y)\rightarrow A$ is a toric log Calabi--Yau pair. 
Thus, Lemma~\ref{lem:toric-vg-g} implies that the general fiber of 
$(Y,\Delta_Y)\rightarrow A$ is a toric log Calabi--Yau pair.
Hence, we can apply Lemma~\ref{lem:div-hor} to deduce that a finite cover $(Y,\Delta_Y)$ of $(X,\Delta)$ admits a 
weak toric log Calabi--Yau fibration to an abelian variety.

Now, we turn to prove that the weak toric log Calabi--Yau fibration 
is indeed a toric log Calabi--Yau fibration.
Note that $Y$ is $\qq$-factorial, 
$\phi_Y \colon Y\rightarrow A$ is a Fano type morphism, 
$A$ is a smooth variety, 
and both the moduli and boundary divisors induced by the canonical bundle formula on $A$ are trivial.
Hence, due to Proposition~\ref{prop:from-weak-to-toric}, it suffices to argue that $\phi_Y\colon Y\rightarrow A$ has no degenerate divisors.
Note that by Lemma~\ref{lem:div-hor}, every irreducible strata of $\Delta_Y$ 
restricts to an irreducible variety in a general fiber.
Note that a general fiber of $(Y,\Delta_Y)$ over $A$ has a $0$-dimensional strata, as it is a toric log Calabi--Yau pair.
Then, the pair $(Y,\Delta_Y)$ has a strata $S$ such that its normalization $S^\nu$ admits a birational morphism onto $A$.
By adjunction, the normal variety $S^\nu$
admits a log Calabi--Yau structure $(S^\nu,B_{S^\nu})$ such that the induced birational morphism $\phi_S \colon S^\nu \rightarrow A$ is crepant for $(S^\nu,B_{S^\nu})$ and $A$.
Since $A$ is smooth, we conclude that $B_{S^\nu}=0$ and $S\rightarrow A$ is an isomorphism.
By~\cite[Lemma 2.6]{MYY24}, up to replacing $f_Y$ with an iteration, the finite morphism induced on $S^\nu$ is int-amplified, so the finite morphism $f_A$ induced on $A$ is int-amplified.

Now, assume that there is a degenerate divisor $S$ of $\phi_Y\colon (Y,\Delta_Y)\rightarrow A$.
Let $Z:=\phi_Y(S)$.
Every irreducible component of $f_A^{-1}(Z)$ is the image of a degenerate divisor on $Y$ (see Lemma~\ref{lem:degenerate-divisors-pullback}).
By Lemma~\ref{lem:int-amplified-abelian-no-fixed-sub}, we know that $f_A$ has no fixed subvarieties. 
Thus, if $\phi_Y\colon Y\rightarrow A$ has one degenerate divisor, then it would have infinitely many degenerate divisors, leading to a contradiction (see Lemma~\ref{lem:finite-deg-div}).
We conclude that $\phi_Y$ has no degenerate divisors,
so $\phi_Y\colon (Y,\Delta_Y)\rightarrow A$ is a toric log Calabi--Yau fibration by
Proposition~\ref{prop:from-weak-to-toric}.
Henceforth, the log Calabi--Yau pair
$(X,\Delta)$ has a finite cover
$(Y,\Delta_Y)$ that admits a toric log Calabi--Yau fibration to an abelian variety.
\end{proof} 

\section{Dynamics on the dual complex}
\label{sec:dynamics}
In this section, we study the dynamics induced
on the dual complex $\mathcal{D}(X,\Delta)$
by an int-amplified endomorphism
$f\colon (X,\Delta) \rightarrow (X,\Delta)$.
Throughout this section, we assume that the dual complex $\mathcal{D}(X,\Delta)$ of a log Calabi--Yau pair $(X,\Delta)$ is simplicial (see, e.g.~\cite[Remark 10]{dFKX17}).
Given a log canonical place $E$ of $(X,\Delta)$, we denote by $v_E$ the corresponding point in the dual complex $\mathcal{D}(X,\Delta)$.

\begin{definition}
{\em 
Let $(X,\Delta)$ be a log Calabi--Yau pair
and $\mathcal{D}(X,\Delta)$ be its dual complex. 
An {\em integral lattice representation}
of $\mathcal{D}(X,\Delta)$ is an embedding $p\colon \mathcal{D}(X,\Delta)\rightarrow \rr^k$ such that all its vertices get mapped into the vertices of a simplex.
The {\em rational points}
of the dual complex, denoted by $\mathcal{D}(X,\Delta)_\qq$, are the rational points in the integral lattice representations.
The rational points of $\mathcal{D}(X,\Delta)$ are precisely those points that correspond to log canonical places of $(X,\Delta)$.
}
\end{definition} 

\begin{definition}\label{def:map-between-dc}
{\em 
Let $(X,\Delta)$ and $(Y,\Delta_Y)$ be two qdlt sub-pairs.
Let $f\colon (X,\Delta)\rightarrow (Y,\Delta_Y)$ be a crepant finite morphism. 
We obtain an induced map $\mathcal{D}(f)\colon \mathcal{D}(\Delta^{=1})\rightarrow \mathcal{D}(\Delta_Y^{=1})$ defined as follows.
Given a prime component $E$ of $\Delta^{=1}$, by Riemann-Hurwitz, it maps onto a component $S$ of $\Delta_{Y}^{=1}$ and we set $\mathcal{D}(f)(v_E)=v_S$. 
Then, the map extends to a continuous function by linearity. 
}
\end{definition}

\begin{theorem}\label{thm:def-D(f)}
Let $f\colon (X,\Delta) \to (X,\Delta)$ be a finite endomorphism of a log Calabi--Yau pair. Then $f$ induces a continuous function 
$\mathcal{D}(f)\colon \mathcal{D}(X,\Delta) \to \mathcal{D}(X,\Delta)$.
Moreover, for every commutative diagram
\[
\xymatrix{
(Y_1,B_1) \ar[r]^-{f_Y} \ar[d] & \ar[d] (Y_2,B_2) \\
(X,\Delta) \ar[r]_-f & (X,\Delta) 
}
\]
with each $(Y_i,B_i)\rightarrow (X,\Delta)$ a qdlt modification and $f_Y$ a crepant finite morphism, the 
continuous function $\mathcal{D}(f)$ 
satisfies the following condition:
If $f_Y(E_1)=E_2$ with $E_1$ a prime component of  $\lfloor B_1\rfloor$, then $\mathcal{D}(f)(v_{E_1})=v_{E_2}$.
Furthermore, we have $\mathcal{D}(f)(\mathcal{D}(X,\Delta)_\qq)=\mathcal{D}(X,\Delta)_\qq$.
\end{theorem} 

\begin{proof}
First, we aim to define the continuous function $\mathcal{D}(f)$.
Fix a qdlt modification $(Z_2,\Gamma_2) \to (X,\Delta)$  and complete the diagram:
\[
\xymatrix{
(Z_1,\Gamma_1) \ar[r]^-{f_Z} \ar[d] & \ar[d] (Z_2,\Gamma_2) \\
(X,\Delta) \ar[r]_-f & (X,\Delta) 
}
\]
by taking the normalization of the fiber product. 
Then, the pair $(Z_1,\Gamma_1)$ is also a qdlt pair. 
The finite map $f_Z$ induces a continuous map $\mathcal{D}(f_Z) \colon \mathcal{D}(Z_1,\Gamma_1) \to \mathcal{D}(Z_2,\Gamma_2)$ (see Definition~\ref{def:map-between-dc}). Fix PL-isomorphisms $\phi_1\colon \mathcal{D}(Z_1,\Gamma_1) \to \mathcal{D}(X,\Delta)$ and $\phi_2\colon \mathcal{D}(Z_2,\Gamma_2) \to \mathcal{D}(X,\Delta)$ (see~\cite[Proposition 11 and Corollary 38]{dFKX17}). Then, we define $\mathcal{D}(f)$ 
to be the unique continuous function making the following diagram commute:
\[
    \xymatrix{
        \mathcal{D}(Z_1,\Gamma_1) \ar[r]^-{\mathcal{D}(f_Z)} \ar[d]_-{\phi_1} & \mathcal{D}(Z_2,\Gamma_2) \ar[d]^-{\phi_2}\\
        \mathcal{D}(X,\Delta) \ar[r]_-{\mathcal{D}(f)} & \mathcal{D}(X,\Delta)
    }
\]
By definition $\mathcal{D}(f)$ is a continuous map between topological spaces which is also a continuous map between PL-pseudomanifolds for two (possibly different) PL-structures on $\mathcal{D}(X,\Delta)$. 

Consider now a commutative diagram 
\[
\xymatrix{
(Y_1,B_1) \ar[r]^-{f_Y} \ar[d] & \ar[d] (Y_2,B_2) \\
(X,\Delta) \ar[r]_-f & (X,\Delta) 
}
\]
as in the statement. Let $(W_2,\Omega_2)$ be a common log resolution of $(Z_2,\Gamma_2) \to (X,\Delta) \leftarrow (Y_2,B_2)$, and define $(W_1,\Omega_1)$ by taking log pull-back of $(X,\Delta)$ to the normalization of the fiber product in the following commutative diagram
\[
\xymatrix{
(W_1,\Omega_1) \ar[r]^-{f_W} \ar[d] & \ar[d] (W_2,\Omega_2) \\
(Y_1,B_1) \ar[r]_-{f_Y} & (Y_2,B_2) 
}
\]
Therefore, we obtain the following commutative diagram of crepant birational morphisms:
\[
    \xymatrix{
        (W_1,\Omega_1) \ar[rr] \ar[rd] \ar[dd] & & (W_2,\Omega_2) \ar[rd] \ar'[d][dd] \\
            & (Z_1,\Gamma_1) \ar[rr] \ar[dd] & & (Z_2,\Gamma_2) \ar[dd]\\
        (Y_1,B_1) \ar'[r][rr] \ar[rd] & & (Y_2,B_2) \ar[rd]&\\
            & (X,\Delta) \ar[rr] & & (X,\Delta).
    }
\]
Then, we obtain a commutative diagram of dual complexes 
\[
    \xymatrix{
        \mathcal{D}(W_1,\Omega_1) \ar[rr]^-{\mathcal{D}(f_W)} \ar[rd]_-{\psi_1} \ar[dd]_-{\pi_1} & & \mathcal{D}(W_2,\Omega_2) \ar[rd]^-{\psi_2} \ar'[d][dd]^-{\pi_2} \\
            & \mathcal{D}(Z_1,\Gamma_1) \ar[rr]^-(.3){\mathcal{D}(f_Z)} \ar[dd]^(0.7){\phi_1} & & \mathcal{D}(Z_2,\Gamma_2) \ar[dd]^-{\phi_2} \\
        \mathcal{D}(Y_1,B_1) \ar'[r][rr]^-{\mathcal{D}(f_Y)}  & & \mathcal{D}(Y_2,B_2)  &\\
            & \mathcal{D}(X,\Delta) \ar[rr]^-{\mathcal{D}(f)} & & \mathcal{D}(X,\Delta)
    }
\]
where the maps $\psi_i$ and $\pi_i$ are PL-isomorphisms which are sequences of collapses and the maps $\mathcal{D}(f_Y)$ and $\mathcal{D}(f_W)$ are defined as in Definition~\ref{def:map-between-dc}.
By construction, we know that $\mathcal{D}(f_Y)(v_{E_1})=v_{E_2}$.
By the commutativity of the diagram, we get
$\mathcal{D}(f_W)(v_{E_1})=v_{E_2}$.
Using the commutativity of the diagram again, we conclude that $\mathcal{D}(f)(v_{E_1})=v_{E_2}$.

The previous paragraph implies that $\mathcal{D}(f)(\mathcal{D}(X,\Delta)_\qq)\subseteq \mathcal{D}(X,\Delta)_\qq$.
Henceforth, in a integral lattice representation of $\mathcal{D}(X,\Delta)$,
the continuous map $\mathcal{D}(f)$
is a piece-wise linear rational function.
Furthermore, 
by construction, we know that $\mathcal{D}(f)^{-1}(v)$ is a finite set for every $v\in \mathcal{D}(X,\Delta)$. 
Even more, this set is not empty, as for any rational point $v_E \in \mathcal{D}(X,\Delta)_\qq$ we can take a qdlt modification $(Z_2,\Gamma_2)$ of $(X,\Delta)$ that extracts $E$
and complete the diagram as in the first paragraph of the proof.
We conclude that the
pre-image via $\mathcal{D}(f)$ of any rational point of $\mathcal{D}(X,\Delta)$
is a finite set of rational points. 
\end{proof} 

\begin{theorem}\label{thm:D(f)-bijection}
Let $(X,\Delta)$ be a log Calabi--Yau pair.
Let $f\colon (X,\Delta)\rightarrow (X,\Delta)$ be an int-amplified endomorphism.
Then, the continuous function $\mathcal{D}(f)$ is a bijection.
\end{theorem} 

\begin{proof}
Assume that $\mathcal{D}(f)$ is not a bijection. As $\mathcal{D}(X,\Delta)$ is a pseudomanifold, we have that $\deg \mathcal{D}(f) > 1$, and hence the degree of $\mathcal{D}(f^n)$ diverges.

By \cite[Theorem 5.1]{Mor21}, if $G$ is a finite subgroup of $\Aut(X,\Delta)$, then there exists a positive integer $c$, only depending on the dimension of $X$, such that $G$ has a normal subgroup $A\trianglelefteq G$ of index at most $c$ that acts trivially on $\mathcal{D}(X,\Delta)$.

Consider $n$ such that $\deg \mathcal{D}(f^n) > c$ and let $g\colon (Y,\Delta_Y) \to (X,\Delta)$ the Galois closure of $f^n$. By Lemma \ref{lem:stand-coeff}, $\Delta$ has standard coefficients, so $(Y,\Delta_Y)$ is a log Calabi--Yau pair and $\Delta_Y$ has standard coefficients. Let $G\leq \Aut(X,\Delta)$ be a finite group such that we have the following commutative diagram

\[
    \xymatrix{
        (Y,\Delta_Y) \ar[d]_-g \ar[r]^-{/G} &  (X/G,\Delta/G) \ar[d]^\cong\\
        (X,\Delta) \ar[r]_-{f^n} & (X,\Delta). 
        }
\]
Then, the size of the orbits of the action of $G$ on $(Y,\Delta_Y)$ is at most $c$, which implies that $\deg \mathcal{D}(f^n) \leq c$ for all $n\in \nn$. Therefore, $\deg \mathcal{D}(f) = 1$ and $\mathcal{D}(f)$ is a bijection.
\end{proof} 

\section{Lifting endomorphisms to a partially dlt modification}
\label{sec:lifting}

In this section, we prove the second main theorem of the article.
We show that an int-amplified endomorphism 
of a log Calabi--Yau pair $(X,\Delta)$ lifts
to a suitable pdlt modification.

\begin{theorem}\label{thm:lifting-pdlt}
Let $X$ be a variety with klt type singularities
and let $(X,\Delta)$ be a log Calabi--Yau pair admitting an int-amplified endomorphism $f$.
Then, there exists a positive integer $n$, a pdlt modification $\pi\colon (Y,\Delta_Y)\rightarrow (X,\Delta)$, and an int-amplified endomorphism $g$ of $(Y,\Delta_Y)$ making the following diagram commutative:
\[
\xymatrix{
(Y,\Delta_Y)\ar[d]_-{\pi} \ar[r]^-{g} & (Y,\Delta_Y)\ar[d]^-{\pi} \\
(X,\Delta)\ar[r]^-{f^n} & (X,\Delta).
}
\]
\end{theorem}

We will need the following notation
and lemma below.

\begin{notation}\label{not:Dual-over-Z}
{\em 
Let $(X,\Delta)$ be a log pair.
Let $Z$ be a log canonical center of $(X,\Delta)$.
We denote by $\mathcal{D}(X,\Delta;Z)$ the subset of $\mathcal{D}(X,\Delta)$
which is the smallest subcomplex
that contains all the log canonical places
whose center is $Z$.
}
\end{notation}

\begin{lemma}\label{lem:collapsible-local}
Let $X$ be a klt type variety and 
let $(X,\Delta)$ be a log Calabi--Yau pair.
Let $Z\subset X$ be a log canonical center.
Then, the simplicial complex $\mathcal{D}(X,\Delta;Z)$ is collapsible.
\end{lemma}

\begin{proof}
Let $f\colon (Y,\Delta_Y) \to (X,\Delta)$ 
be a dlt modification,
and let $P_1,\ldots, P_t$ be the divisors in $Y$ whose center is $Z$.
As $X$ is of klt type, 
there exists an effective divisor $B$ on $X$
such that $(X,B)$ is klt,
and in particular, for $0 < \epsilon \ll 1$,
the pair $(X,(1-\epsilon)B + \epsilon \Delta)$
is klt.
Call $\Gamma := (1-\epsilon)B + \epsilon \Delta$. 
By ~\cite[Theorem 1]{Mor19}, there exists a
projective birational morphism 
$g\colon (Y',\Gamma_{Y'}) \to (X,\Gamma)$ 
that extracts exactly $P_1,\ldots, P_t$ 
with coefficients in $(0,1)$.

Consider the divisor $E$ on $Y'$
such that $(\Gamma_{Y'}+E)^{=1} = \sum_{i=1}^t P_i$.
The pair $(Y',\Gamma_{Y'} + E)$ is log canonical
and, by definition, we have
$\mathcal{D}(Y',\Gamma_{Y'}+E) = \mathcal{D}(X,\Delta;Z)$.

Notice that $K_{Y'} + \Gamma_{Y'} + E \sim_{\qq,X} E$,
and $E$ is effective and exceptional. 
Then the $(K_{Y'} + \Gamma_{Y'} + E)$-MMP 
over $X$ terminates, say
$\phi \colon (Y',\Gamma_{Y'} + E) \dasharrow (W,\Gamma_W)$, 
and it contracts all the prime components of $E$.
Hence, $\phi_*\left((\Gamma_{Y'}+E)^{=1}\right) = 0$,
and by ~\cite[Lemma 23 and Theorem 19]{dFKX17} 
this implies that 
$\mathcal{D}(X,\Delta;Z) = \mathcal{D}(Y',\Gamma_{Y'}+E)$
collapses to a point.
Thus, $\mathcal{D}(X,\Delta;Z)$ is collapsible.
\end{proof}

\begin{proof}[Proof of Theorem~\ref{thm:lifting-pdlt}]
Assume that $(X,\Delta)$ is not a pdlt pair, and let $Z$ be a non-pdlt center.
First, notice that for $n\gg 0$, 
$f^n(Z) = Z$
as $X$ has finitely many non-pdlt centers,
and $f$ sends log canonical centers to 
log canonical centers.
We replace $f$ with $f^n$.
From Lemma~\ref{lem:collapsible-local}, 
we know that the set $\mathcal{D}(X,\Delta;Z)$ 
is a finite collapsible simplicial complex.
Choosing an integral lattice representation of $\mathcal{D}(X,\Delta;Z)$, we see that $\mathcal{D}(f)$ is a piecewise linear rational function of a collapsible finite simplicial complex. 
Lefschetz theorem implies that $\mathcal{D}(f)$ has a fixed point in $\mathcal{D}(X,\Delta;Z)$ and the fact that $\mathcal{D}(f)$ is a rational piecewise linear function implies that we can find a rational fixed point.
Thus, there is a $v_E\in \mathcal{D}(X,\Delta;Z)$ that is fixed under the action of $\mathcal{D}(f)$.
Let $\pi_1\colon (Y_1,\Delta_{Y_1}) \to (X,\Delta)$ be
a projective birational morphism that only extracts $E$ (see e.g.,~\cite[Theorem 1]{Mor19}).
We complete the diagram 
\begin{equation}\label{eq:diagram-lift-pdlt}
    \xymatrix{
    (Y_2,\Delta_{Y_2}) \ar[d]_{\pi_2} \ar[r]^{g_2} & (Y_1,\Delta_{Y_1}) \ar[d]^{\pi_1}\\
    (X,\Delta) \ar[r]_{f} & (X,\Delta)
    }
\end{equation}
by  taking normalization of the fiber product. We claim that $(Y_2,\Delta_{Y_2}) \to (X,\Delta)$ also only extracts $E$. 

Let $F$ be a divisor in $(Y_2,\Delta_{Y_2})$ whose center is $Z$, and so $g_2(F) = E$.
Take $(Z_1,\Delta_{Z_1}) \to (Y_1,\Delta_{Y_1})$ 
a qdlt modification, and complete the diagram to obtain:
\[
    \xymatrix{
    (Z_2,\Delta_{Z_2}) \ar[r]^{g_2'} \ar[d] & (Z_1,\Delta_{Z_1}) \ar[d]\\
    (Y_2,\Delta_{Y_2}) \ar[d]_{\pi_1} \ar[r]^-{g_2} & (Y_1,\Delta_{Y_1}) \ar[d]^{\pi_1}\\
    (X,\Delta) \ar[r]_{f} & (X,\Delta).
    }
\]
Then, $g_2'(F) = E$, but the diagram 
\[
    \xymatrix{
    (Z_2,\Delta_{Z_2}) \ar[r]^{g_2'} \ar[d] & (Z_1,\Delta_{Z_1}) \ar[d]\\
    (X,\Delta) \ar[r]_{f} & (X,\Delta)
    }
\]
satisfies the conditions of Theorem \ref{thm:def-D(f)}, 
hence $\mathcal{D}(f)(v_F) = v_E$. 
By Theorem \ref{thm:D(f)-bijection}, 
the map $\mathcal{D}(f)$ is a bijection, 
so $v_F = v_E$, implying that $F = E$.

Let $\pi_m\colon (Y_m,\Delta_m) \to (X,\Delta)$ 
be the morphism obtained, 
as in diagram~\eqref{eq:diagram-lift-pdlt}, by
\[
    \xymatrix{
     (Y_m,\Delta_{Y_m}) \ar[d]_{\pi_m} \ar[r]^{g_m} & (Y_1,\Delta_{Y_1}) \ar[d]^{\pi_1}\\
     (X,\Delta) \ar[r]_{f^m} & (X,\Delta),
    }
\]
where, for all $m \geq 1$, $\pi_m\colon (Y_m,\Delta_{Y_m}) \to (X,\Delta)$ only extracts $E$.

For $m\geq 1$, let $(Y'_m,\Delta_{Y'_m}) \to (Y_m,\Delta_{Y_m})$ be a small $\qq$-factoralization. Then, as the divisor $-E$ is big and nef over $X$, 
the morphism $Y'_1 \to X$
is a relative Mori Dream Space.
In particular, as mentioned in the proof of Theorem \ref{thm:lift-to-Q-fact}, there exists $n \in \nn$ such that for infinitely many $m\in \nn$, $(Y'_m,\Delta_{Y'_m})$
is isomorphic to $(Y'_n,\Delta_{Y'_n})$.
Therefore, we may assume, 
by possibly replacing $(Y_1,\Delta_{Y_1})$ with $(Y_n,\Delta_{Y_n})$ and
$f$ with a higher power,
that there exists a $\qq$-factorial pair 
$(Y',\Delta_{Y'})$ such that 
$(Y',\Delta_{Y'}) \to (Y_m,\Delta_m)$ 
is a small $\qq$-factoralization 
for all $m$.

\[
    \xymatrix{
        (Y',\Delta_{Y'}) \ar[d] & (Y',\Delta_{Y'}) \ar[d]\\
        (Y_m,\Delta_{Y_m}) \ar[d]_{\pi_m} \ar[r]^{g_m} & (Y_1,\Delta_{Y_1}) \ar[d]^{\pi_1}\\
        (X,\Delta) \ar[r]_{f^m} & (X,\Delta)
    }
\]

As the morphism $Y' \to X$ 
is a relative Mori Dream Space,
it has only finitely many 
intermediate contractions 
$Y'\rightarrow Y \rightarrow X$.
Hence, there exists $n,m \in \nn$, such that the pairs
$(Y_{n},\Delta_{Y_{n}})$ must be isomorphic 
to $(Y_m,\Delta_{Y_m})$.

Therefore, replacing $(Y_{1},\Delta_{Y_{1}})$ with $(Y_{n},\Delta_{Y_{n}})$, we obtain the commutative diagram 
\[
    \xymatrix{
     (Y_1,\Delta_{Y_1}) \ar[d]_{\pi_1} \ar[r]^{g} & (Y_1,\Delta_{Y_1}) \ar[d]^{\pi_1}\\
     (X,\Delta) \ar[r]_{f^n} & (X,\Delta),
    }
\]
where $g$ is a finite morphism.
Finally, we need to prove that 
$g$ is int-amplified. 
Let $A$ be an ample divisor on $X$ such that
$(f^n)^* A - A$ is ample. Then $\pi_1^* A = B$ is big and nef, and so
\[
    g^* B - B = g^*\pi_1^* A - \pi_1^* A =  \pi_1^* (f^n)^*A - \pi_1^* A = \pi_1^*((f^n)^*A - A)
\]
is big. Then, $g^* B - B$ is big 
and $g$ is int-amplified 
by \cite[Theorem 3.3]{Men17}.

The number of non-pdlt centers is finite, 
and $(Y_1,\Delta_{Y_1})$ has strictly less 
non-pdlt centers than $(X,\Delta)$.
Then, after finitely many iterations of the
above procedure, we obtain a 
pdlt modification $\pi\colon (Y,\Delta_Y) \to (X,\Delta)$,
along with an int-amplified endomorphism 
$g\colon (Y,\Delta_Y) \to (Y,\Delta_Y)$ such that
$\pi \circ g = f^n\circ \pi$ for some $n\gg 0$.
\end{proof}

\begin{proof}[Proof of Theorem~\ref{introthm:3}]
Let $X$ be a klt type variety.
Let $(X,\Delta)$ be a log Calabi--Yau pair
admitting an int-amplified endomorphism $f$.
By Theorem~\ref{thm:lifting-pdlt}, up to replacing $f$ with an iteration $f^n$, we have a commutative diagram:
\[
\xymatrix{
(Y,\Delta_Y)\ar[d]_-{\pi} \ar[r]^-{g} &(Y,\Delta_Y)\ar[d]^-{\pi} \\
(X,\Delta) \ar[r]^-{f}  & (X,\Delta),
}
\]
where $\pi$ is a pdlt modification
and $g$ is an int-amplified endomorphism.
By Theorem~\ref{thm:main-pdlt}, we know that $(Y,\Delta_Y)$ is the finite quotient of
a toric log Calabi--Yau fibration over an abelian variety.
\end{proof}

\section{Examples and questions}

In this section, 
we collect some examples
and pose some questions for further research.
First, we show that Theorem~\ref{thm:main-pdlt}
is not valid if we drop the pdlt condition.
In particular, Theorem~\ref{introthm:pol-endo-log-CY} is not valid if we drop the dlt condition.

\begin{example}\label{ex:not-pdlt}
{\em 
Let $\lambda\in \pp^1\setminus \{0,1,\infty\}$.
Then, for each $n\geq 2$, the log Calabi--Yau pair
$(\pp^1,\frac{1}{2}(\{0\}+\{1\}+\{\lambda\}+\{\infty\}))$
admits a polarized endomorphism $f_n$ of degree $n^2$.
Note that $f_n^*\mathcal{O}(1)=\mathcal{O}(n^2)$, so
$f_n$ lifts to a polarized endomorphism $g_n$
of $X:={\rm Bl}_p(\pp^2)\simeq \pp_{\pp^1}(\mathcal{O}\oplus \mathcal{O}(1))$. 
By construction, the polarized endomorphism $g_n$ 
is a polarized endomorphism of the log Calabi--Yau pair
$(X,\Delta):=(X,S_0+S_\infty+\frac{1}{2}(F_0+F_1+F_\lambda+F_\infty))$,
where $S_0$ and $S_\infty$ are the torus invariant sections over $\pp^1$ and the $F_i$'s are fibers over $\pp^1$.
We may assume that $S_0^2=1$ and $S_\infty^2=-1$.
Note that $g_n^*S_0=n^2S_0$.
Thus, the polarized endomorphism $g_n$ descends
to a polarized endomorphism $h_n$ of the ample model of $S_0$, which is $\pp^2$.
The push-forward of $\Delta$ to $\pp^2$ gives an effective divisor $B$ such that the polarized endomorphism $h_n$ is a polarized endomorphism of the log Calabi--Yau pair $(\pp^2,B)$.
Note that $(\pp^2,B)$ does not have pdlt singularities, indeed the point $p\in \pp^2$ is a log canonical center
which is not contained in $\lfloor B\rfloor$.

We argue that no pdlt modification of $(\pp^2,B)$ admits a toric log Calabi--Yau fibration to an elliptic curve.
Note that for any pdlt modification $(Y,B_Y)$ of $(\pp^2,B)$ the boundary divisor $B_Y$ has a component with coefficient $\frac{1}{2}$.
This component cannot be vertical nor horizontal over the elliptic curve.
Thus, the pair $(\pp^2,B)$ has no pdlt modification 
that admits a toric log Calabi--Yau fibration over an abelian variety.

Let $(Y,B_Y)\rightarrow (\pp^2,B)$ be a crepant finite cover of $(\pp^2,B)$, then $(Y,B_Y)$ is not pdlt (see Lemma~\ref{lem:pdlt-descends-finite}).
In particular, $(Y,B_Y)$ does not admit a toric log Calabi--Yau fibration over an abelian variety.

By the two previous paragraphs, we conclude that we cannot produce a toric log Calabi--Yau fibration structure
from $(\pp^2,B)$ by only taking finite covers
or by only taking pdlt modifications. 
However, by Theorem~\ref{introthm:3}, we know that such a structure can be found by 
performing a pdlt modification
followed by a finite cover.
Indeed, the pair $(X,\Delta)$ considered above is already pdlt.
We have a degree $2$ crepant finite morphism $E\rightarrow (\pp^1,\frac{1}{2}(\{0\}+\{1\}+\{\lambda\}+\{\infty\})$ where $E$ is an elliptic curve.
Let $Y:=E \times_{\pp^1} X$ 
and let $(Y,\Delta_Y)$ be the log pull-back of $X$ to $Y$.
Then, the fibration $f_Y\colon (Y,\Delta_Y)\rightarrow E$ is a toric log Calabi--Yau fibration.
Thus, we obtain a commutative diagram 
\[
\xymatrix{
(Y,\Delta_Y) \ar[r]^-{g_X}\ar[d]_-{f_Y} & (X,\Delta) \ar[r]^-{\pi}\ar[d]^-{f} & (\pp^2,B)\\
E \ar[r]_-{g} & (\pp^1,\frac{1}{2}\left(\{0\}+\{1\}+\{\lambda\}+\{\infty\}\right)) & 
}
\]
where $\pi$ is a pdlt modification, 
$g_X$ is a crepant finite morphism, and
$f_Y\colon (Y,\Delta_Y) \rightarrow E$ is a 
toric log Calabi--Yau fibration. 
}
\end{example}

\begin{example}\label{ex:index-one-not-klt}
{\em 
For every integer $n$, we provide a simple example of a $2$-dimensional log Calabi--Yau pair of index $n$ which is not pdlt.

Let $L$ be a line in $\pp^2$ and $p$ be a point not contained in $L$.
Let $L_1,\dots,L_{2n}$ be general lines containing $p$.
Then, the pair $(\pp^2, L+\frac{1}{n}(L_1+\dots+L_{2n}))$ is a log Calabi--Yau pair of index $n$.
Note that $p$ is a log canonical center. 
Indeed, the divisor obtained by blowing up $p$ 
is a log canonical place.
On the other hand, we have
$\lfloor L+\frac{1}{n}(L_1+\dots+L_{2n})\rfloor =L$ which does not contain $p$.
We conclude that $(\pp^2,L+\frac{1}{n}(L_1+\dots+L_{2n}))$ is a log Calabi--Yau pair
that is not pdlt.
}
\end{example}

\begin{example}\label{ex:no-lift-resol}
{\em 
We show an example of a variety $X$ admitting an int-amplified endomorphism
such that no resolution of singularities of $X$ admits an int-amplified endomorphism.

Let $X$ be a rational klt surface 
with $\kappa(X,-K_X)=0$ admitting an int-amplified endomorphism.
Examples of these surfaces can be found in the works of Tokunaga and Yoshihara ~\cite[Example 1.1]{TY95} and Matsuzawa and Yoshikawa ~\cite[Theorem 1.2.(3)]{MY21}).
Let $Y\rightarrow X$ be a resolution of singularities.
Assume that $Y$ admits an int-amplified endomorphism.
Note that $Y$ is a smooth rational surface.
Then, would imply that $Y$ is a smooth toric variety (see e.g.,~\cite[Theorem 1.1]{FN05}).
Since $Y\rightarrow X$ is a contraction, we would get that $X$ is a toric variety as well,
and so $-K_X$ is big.
This contradicts the fact that $\kappa(X,-K_X)=0$.
Thus, we conclude that no 
resolution of $X$ admits an int-amplified endomorphism.
}
\end{example}

Example~\ref{ex:not-pdlt} shows that a log Calabi--Yau pair
$(X,\Delta)$ may admit no finite cover that is a toric log Calabi--Yau fibration over an abelian variety.
However, in that example, the underlying variety $X$ is itself toric. 
Thus, the folklore conjecture is still open in this setting.
We propose the following question.

\begin{question}
Let $X$ be a klt type variety and
let $(X,\Delta)$ be a log Calabi--Yau pair admitting a polarized endomorphism.
Is $X$ a finite quotient of a toric fibration over an abelian variety?
\end{question} 

A common theme of this article
and~\cite{MYY24} 
is reducing the study of polarized endomorphisms
to the setting of Galois polarized endomorphisms.
To do so, in both articles, we study 
polarized endomorphisms that respect certain log 
Calabi--Yau pair structures $(X,\Delta)$ on the variety $X$.
However, it may be possible to reduce 
to the setting of Galois polarized endomorphisms
by using deformations instead.
We propose the following question in this direction.

\begin{question}
Let $f\colon X\rightarrow X$ be a polarized endomorphism
of a klt variety $X$.
Can we find a flat family $f_t\colon X\rightarrow X$ of polarized endomorphisms such that $f_1=f$
and $f_0$ is a Galois polarized endomorphism?
\end{question} 

The previous question is motivated by the toric case
in which we can degenerate a polarized endomorphism
to one defined by monomials in the ambient variety.

\bibliographystyle{habbvr}
\bibliography{references}

\end{document}